\documentclass[a4paper]{amsart}

\usepackage[utf8]{inputenc}
\usepackage{tikz}
\usetikzlibrary{cd}
\usepackage{amssymb}
\usepackage[]{hyperref}
\hypersetup{
    colorlinks,
    linkcolor={red!50!black},
    citecolor={blue!50!black},
    urlcolor={blue!80!black}
}
\usepackage{cleveref} % load --after-- hyperref

\newcommand{\on}{\operatorname}
\newcommand{\ov}{\overline}
\newcommand{\wt}{\widetilde}

\newcommand{\ds}{\displaystyle}
\newcommand{\cal}{\mathcal}

\newcommand{\ntoi}[1]{\underset{n\to\infty}{#1}}

\newcommand{\tend}{\longrightarrow}

\newtheorem{theorem}{\bf Theorem}
\newtheorem{proposition}[theorem]{\bf Proposition}
\newtheorem{sublemma}[theorem]{\bf Sublemma}
\newtheorem{lemma}[theorem]{\bf Lemma}
\newtheorem*{lemma*}{\bf Lemma}
\newtheorem*{claim*}{\bf Claim}
\newtheorem{corollary}[theorem]{\bf Corollary}
\newtheorem{compl}[theorem]{\bf Complement}
\theoremstyle{remark}
\newtheorem{remark}[theorem]{\bf Remark}
\newtheorem*{remark*}{\bf Remark}
\newtheorem*{reminder}{\bf Reminder}
\theoremstyle{definition}
\newtheorem{definition}[theorem]{\bf Definition}
\newtheorem{notation}[theorem]{\bf Notation}
\newtheorem{condition}[theorem]{\bf Condition}

\newcommand{\C}{{\mathbb C}}
\newcommand{\D}{{\mathbb D}}
\renewcommand{\H}{{\mathbb H}}
\newcommand{\Q}{{\mathbb Q}}
\newcommand{\Z}{{\mathbb Z}}
\newcommand{\R}{{\mathbb R}}

\newcommand{\N}{{\mathbb N}}
\renewcommand{\S}{{\mathbb S}}

\renewcommand{\a}{\alpha}
\newcommand{\eps}{\epsilon}

\renewcommand{\epsilon}{\varepsilon}
\renewcommand{\phi}{\varphi}
\renewcommand{\emptyset}{\varnothing}

\newcommand{\dom}{\on{dom}}

\newcommand{\im}{\Im}
\newcommand{\re}{\Re}
\renewcommand{\Im}{\on{Im}}
\renewcommand{\Re}{\on{Re}}

\newcommand{\setof}[2]{\big\{{#1}\,;\,{#2}\big\}}

\begin{document}

\title{Smooth Siegel disks everywhere} \subjclass{}
\begin{author}[A.~Avila]{Artur Avila}
\address{ %
  UPMC / IMPA}
\end{author}
\begin{author}[X.~Buff]{Xavier Buff}
\address{ 
  Université Paul Sabatier\\
  Institut de Mathématiques de Toulouse \\
  118, route de Narbonne \\
  31062 Toulouse Cedex \\
  France }
\end{author}
\begin{author}[A.~Chéritat]{Arnaud Chéritat}
\address{ 
  Université Paul Sabatier\\
  Institut de Mathématiques de Toulouse \\
  118, route de Narbonne \\
  31062 Toulouse Cedex \\
  France }
\end{author}

\begin{abstract}
We prove the existence of Siegel disks with smooth boundaries in most families of holomorphic maps fixing the origin. The method can also yield other types of regularity conditions for the boundary.
The family is required to have an indifferent fixed point at $0$, to be parameterized by the rotation number $\alpha$, to depend on $\alpha$ in a Lipschitz-continuous way, and to be non-degenerate. A degenerate family is one for which the set of non-linearizable maps is not dense. We give a characterization of degenerate families, which proves that they are quite exceptional.
\end{abstract}

\maketitle

\tableofcontents

\section*{Introduction}

In \cite{PM}, Pérez-Marco was the first to prove the existence of
univalent maps $f:\D\to \C$ having Siegel disks compactly contained in $\D$ and with smooth
($C^\infty$) boundaries. In \cite{ABC},
we gave a different proof of this result,
showing the existence of quadratic polynomials having Siegel
disks with smooth boundaries (for a simplification of the proof, see
\cite{G}). In \cite{BC}, two authors of the present article
proved the existence of
quadratic polynomials having Siegel disks whose boundaries have any
prescribed regularity between $C^0$ and $C^\infty$ (the method covers other classes of regularity).

In this article, we generalize these results to most families of
maps having a non persistent indifferent cycle. 

\begin{definition}[Non-degenerate families]\label{def:dege}
Assume $I\subset \R$ is an open interval.
Consider a family of holomorphic maps $f_\alpha:\D\to \C$ parameterized by $\alpha \in I$, with \[f_\alpha(z) = e^{2\pi i\alpha}z+\cal O(z^2)\]
and assume that $f_\alpha$ depends continously\footnote{This means: $(\alpha,z)\mapsto f_\alpha(z)$ is continuous.} on $\alpha$.
We say that the family is \emph{degenerate} if the set $\{\alpha\in I~;~f_\alpha\text{ is not linearizable}\}$ is not dense in $I$. Otherwise it is called \emph{non-degenerate}.
\end{definition}

This definition is purely local so if we are given a holomorphic dynamical system on a Riemann surface, we can extend the definition above by working in a chart and restricting the map to a neighborhood of the fixed point.

For example, if $f_\alpha$ is a family of rational maps of the same degree $d\geq 2$, then it is automatically non-degenerate. Indeed, a fixed point of a rational map of degree $\geq 2$ whose multiplier is a root of unity is never linearizable.\footnote{This is a simple and classical fact, that seems difficult to find in written form. If an iterate of $f$ is conjugate on an open set $U$ to a finite order rotation then a further iterate of $f$ is the identity on $U$. Since $f^n$ is holomorphic on the Riemann sphere, it is the identity everywhere. This contradicts the fact that $f^n$ has degree $d^n>1$.}

In \Cref{app} we characterize degenerate families in the case where the dependence with respect to the parameter $\alpha$ is analytic.

\begin{notation}\label{def:KD}
Assume $f:\D\to \C$ is a holomorphic map having an indifferent fixed
point at $0$. We write
\begin{itemize}
\item $K(f)$ the set of
points in $\D$ whose forward orbit remains in $\D$ and
\item $\Delta(f)$ the connected component of the interior of $K(f)$
that contains $0$; $\Delta(f)=\emptyset$ if there is no such component.
\end{itemize}
\end{notation}

\begin{remark*}[Siegel disks] If $\Delta(f)\neq \emptyset$ it is known that $\Delta(f)$ is simply connected\footnote{This is another classical fact. See \Cref{foot:one} in \Cref{sub:sd}.} and that the restriction
$f:\Delta(f)\to \Delta(f)$ is analytically conjugate to a rotation via a conformal bijection
between $\Delta(f)$ and $\D$ sending $0$ to $0$, see \Cref{sub:proplin}.
The set $\Delta(f)$ is usually called a \emph{Siegel disk} in the case $\alpha\notin\Q$ and we will use the same terminology in this article for the case $\alpha\in\Q$, though subtleties arise. See \Cref{sub:sd}
\end{remark*}

We prove here that the main theorem in \cite{BC} holds for a non-degenerate family under the assumption that the dependence $\alpha\mapsto f_\alpha$ is Lipschitz.\footnote{By this we mean: $(\exists C>0)$ $(\forall \alpha\in I$, $\alpha'\in I,$ $z\in\D)$, $|f_{\alpha'}(z)-f_\alpha(z)|\leq C|\alpha'-\alpha|$.}  We thus get in particular (see \Cref{app:grl} for the general statement):

\begin{theorem}\label{thm:main}
Under the non-degeneracy assumptions of \Cref{def:dege}, if moreover the dependence $\alpha\mapsto f_\alpha$ is Lipschitz then
\begin{itemize}
\item $\exists \alpha\in I\setminus\Q$ such that $\Delta(f_\alpha)$ is compactly
contained in $\D$ and $\partial \Delta(f_\alpha)$ is a $C^\infty$ Jordan
curve;

\item $\exists \alpha\in I\setminus\Q$ such that $\Delta(f_\alpha)$ is compactly
contained in $\D$ and $\partial \Delta(f_\alpha)$ is a Jordan curve but
is not a quasicircle;

\item $\forall n\geq 0$, $\exists \alpha\in I\setminus\Q$,
$\Delta(f_\alpha)$ is compactly contained in $\D$ and $\partial
\Delta(f_\alpha)$ is a Jordan curve which is $C^n$ but not $C^{n+1}$.
\end{itemize}
\end{theorem}

A family satisfying the assumptions on the interval $I$ also satisfies them on every sub-interval. It follows that the parameters $\alpha$ in the theorem above are in fact dense in $I$.

\begin{remark*} If $f_\alpha$ is a restriction of another map $g_\alpha$ and $\Delta(f_\alpha)\Subset\D$, then $\Delta(f_\alpha) =\Delta(g_\alpha)$, see \Cref{sub:sd}. So the result gives information on the Siegel disks not only of maps from $\D$ to $\C$ but in fact of any kind of analytic maps, for instance polynomials $\C\to\C$, rational maps $\S\to\S$, entire maps $\C\to\C$, \ldots
\end{remark*}

\begin{remark*}
The main tool for \Cref{thm:main} is Yoccoz's sector renormalization as in several of our previous works (except \cite{A} who uses Risler's work instead \cite{R}). In \cite{A}, \cite{ABC} and \cite{BC} it was crucial to have a family for which it is known that $f_\alpha$ is linearizable if and only if $\alpha$ is a Brjuno number. The progress here is to get rid of this assumption.
\end{remark*}

\section{Conformal radius, straight combs and the general construction.}\label{sec:grl}

The method that Buff and Chéritat first developped to get smooth Siegel disks is one of the offsprings of a fine control, initiated in \cite{C}, on the periodic cycles that arise when one perturbs parabolic fixed points. Still today we can only make it work in specific contexts, which includes quadratic polynomials for instance. With the smooth Siegel disk objective in mind, Avila was able in \cite{A} to identify essential sufficient properties so as to allow for a partial generalization, and also pointed to the bottleneck for a complete generalization. In this section, we essentially follow the presentation in \cite{A}. We also mention a connection with continuum theory.

\medskip

In this whole section, except \Cref{sub:sd}, we consider a non-empty open interval $I\subset \R$ and a continuous family of analytic maps $f_\alpha:\D\to \C$ parameterized by $\alpha\in I$ with $f_\alpha(z)=e^{2\pi i\alpha}z+\cal O(z^2)$.

\subsection{Siegel disks and restrictions}\label{sub:sd}

Given a one dimensional complex manifold $S$ and a holomorphic map
\[f : U\to S\]
defined on an open subset $U$ of $S$, assume there is a neutral fixed point $a\in U$ of multiplier $e^{2\pi i\alpha}$ with 
$\alpha\in\R$.
Call \emph{rotation domain} any open set containing $a$ on which the map is analytically conjugate to a rotation on a Euclidean disk or on the plane or on the Riemann sphere.\footnote{The last case is extremely specific, for we must have $U=S$ isomorphic to the Riemann sphere and $f$ is a rotation.}
If $\alpha\notin\Q$ then the rotation domains are totally ordered by inclusion. This is \emph{never} the case if $\alpha\in\Q$. If $\alpha\notin\Q$ there is a maximal element for inclusion, called \emph{the} Siegel disk\footnote{In the case of a rotation on the Riemann sphere this name is not appropriate since the disk is a sphere\ldots} of $f$ at point $a$. If $\alpha\in\Q$ existence of a maximal element may also fail, depending on the situation. If $\alpha\notin\Q$ the Siegel disk of a restriction is automatically a subset of the original Siegel disk. If $\alpha\in\Q$ this may fail. 

\begin{remark*}
The right approach in the general case is probably to use the \emph{Fatou set}. Here is not the place for such a treatment, so we only give results specific to our situation 
\end{remark*}

In the sequel we assume $S=\C$ and $U$ is bounded and simply connected.

\medskip

We let $K$ be the set of points whose orbit is defined for all times. The set $K\subset U$ is not necessarily closed in $U$. We let $\Delta(f)=\Delta\subset U$ where $\Delta$ is the connected component containing $a$ of the interior $K^{\circ}$, or $\Delta = \emptyset$ if $a\notin K^{\circ}$.
Then $\Delta$ is necessarily simply connected: this is one classical application of the maximum principle.\footnote{\label{foot:one}If $\Delta$ would not be simply connected then there would exist a bounded closed set $C\neq \emptyset$ (not necessarily connected) such that $\Delta\cap C=\emptyset$ and such that $\Delta':=\Delta(f)\cup C$ is open and connected (this is a theorem in planar topology). By the maximum principle, $f^k(\Delta')\subset \D$ for all $k$. Hence $\Delta'$ would be an open subset of $K(f)$, contradicting the definition of $\Delta$.} Any rotation domain for $a$ is necessarily contained in $K$.
It follows that any rotation domain for $a$ is in fact contained in $\Delta$.
Moreover, let us prove that $\Delta$ itself is a rotation domain:
\begin{proof}
First note that we have $f(\Delta)\subset \Delta$. The set $\Delta$ is conformally equivalent to the unit disk $\D$. Conjugate $f$ by a 
conformal map from $\Delta$ to $\D$ sending $a$ to $0$. We get a holomorphic self-map of $\D$ with a neutral fixed point at its center. By the case of equality in Schwarz's lemma this self-map is a rotation.
\end{proof}

\begin{corollary}\label{cor:subdisk}
(We do not make an assumption on $\alpha$.) Let $U'$ be an open subset of $\C$.
Let $g:U'\to \C$ be holomorphic with a neutral fixed point $a$.
Assume $U$ is an open subset of $U'$ containing $a$ and let $f$ be the restriction of $g$ to $U$. Then $\Delta(f)\subset\Delta(g)$.
If moreover $U$ and $U'$ are simply connected and if $\Delta(f)$ is compactly contained in the domain of definition of $f$ then $\Delta(g)=\Delta(f)$. 
\end{corollary}

\begin{proof}
The first claim is immediate.
For the second claim when $\Delta(f)$ is compactly contained in $U$, consider the image of $\Delta(f)$ by the uniformization $(\Delta(g),0)\to (\D,0)$: we get a simply connected subset $A$ of $\D$, invariant by the rotation. In the case $\alpha\notin\Q$ this has to be a disk $B(0,r)$ with $r\leq 1$. If $\alpha\in\Q$, more sets are possible. In any case if $A$ is not equal to $\D$ we can construct a connected invariant open subset of $K(f)$ that strictly contains $\Delta(f)$, leading to a contradiction.
\end{proof}

In the case $U=\D$ our definition of $\Delta$ coincides with \Cref{def:KD}.

\subsection{Properties of the conformal radius as a function of the angle}

Recall that we consider a non-empty open interval $I\subset \R$ and a continuous family of analytic maps $f_\alpha:\D\to \C$ parameterized by $\alpha\in I$ with $f_\alpha(z)=e^{2\pi i\alpha}z+\cal O(z^2)$.
The set $\Delta(f_\alpha)$ has been defined in \Cref{def:KD}, and \Cref{sub:sd} gave a mild generalization and basic properties.

\begin{definition}\label{def:confrad}
We let $r(\alpha)$ denote the conformal radius at $0$ of $\Delta(f_\alpha)$ if it is not empty. Otherwise we set $r(\alpha)=0$.
\end{definition}

Recall that the \emph{conformal radius} of a simply connected open subset $U$ of $\C$ at a point $a\in U$ is defined as the unique $r\in(0,+\infty]$ such that there exists a conformal bijection $\phi : B(0,r) \to U$ with $\phi(0)=a$ and $\phi'(0) = 1$ .

\begin{proposition}\label{prop:basic}
Let $\cal B$ denote the set of Brjuno numbers.\footnote{See \cite{Y} for a definition.}
The function $\alpha\mapsto r(\alpha)$ has the following properties:
\begin{enumerate}
\item It is upper semi-continuous: $\forall \alpha\in\R$, $\limsup_{x\to\alpha} r(x) \leq r(\alpha)$.
\item It takes positive values at Brjuno numbers: $\alpha\in \cal B \implies r(\alpha)>0$.
\item\label{item:basic:3} It is weakly lower semi-continuous on the left and on the right at every Brjuno number: $\alpha \in \cal B$ $\implies$ $\limsup_{x\to\alpha^-} r(x) \geq r(\alpha)$ and $\limsup_{x\to\alpha^+} r(x) \geq r(\alpha)$.
\end{enumerate}
Weak lower semi-continuity on each side can be rephrased as follows: there exists $\alpha_n <\alpha<\alpha'_n$ with $\alpha_n\to\alpha$ and $\alpha'_n\to\alpha$ such that $\lim r(\alpha_n) \geq r(\alpha)$ and $\lim r(\alpha'_n) \geq r(\alpha)$. Since $f$ is upper semi-continuous, these limits are in fact equal to $r(\alpha)$.
\end{proposition}

The first property is classical, see \Cref{lem:usc}.
The second is Brjuno's theorem, \cite{Br,Ru,Y}.
The third follows from Risler's work \cite{R} or from a fine study of Yoccoz's renormalization \cite{ABC}.
According to the method by Buff and Chéritat explained in \cite{ABC}, to get $\liminf r(\alpha_n)\geq r(\alpha)$ it is enough to take a sequence $\alpha_n$ such that $\Phi(\alpha_n)\to \Phi(\alpha)$ where $\Phi$ denotes Yoccoz's variant of the Brjuno sum.
Both references also imply:

\begin{compl}\label{compl:bdd}
In \Cref{item:basic:3} above, one can take sequences $\alpha_n$ and $\alpha'_n$ that are bounded type numbers (Diophantine of order $2$).
\end{compl}

For instance if $\alpha = a_0+1/(a_1+1/\ldots) = [a_0; a_1,\ldots]$ is the continued fraction expansion of $\alpha\in \cal B$ then the sequences $(\theta_{2n})$ and $(\theta_{2n+1})$ work, where $\theta_n= [a_0;a_1, \ldots, a_n, 1+a_{n+1}, 1, 1, 1, \ldots]$. Indeed $\theta_n\to \alpha$, alternating on each side of $\alpha$ and one can check that $\Phi(\theta_n) \to \Phi(\alpha)$, as follows for instance from the remark after Proposition 2 in \cite{ABC}.

Note that in the particular case of a family $f_\alpha$ that depends on $\alpha$ in a Lipschitz-continuous way, \Cref{compl:bdd} is also a corollary of the main lemma of the present article, \Cref{lem:main}.

\subsection{Properties of the linearizing map}\label{sub:proplin}

Consider $f_\alpha$ and recall that $K(f_\alpha)$ is defined as the set of points whose orbit stays in $\D$. Assume that the interior of $K(f_\alpha)$ contains $0$ and recall that $\Delta = \Delta(f_\alpha)$ is defined as the connected component containing $0$ of the interior of $K(f_\alpha)$. Recall that $\Delta$ is simply connected. Any uniformization $\phi:r\D\to\Delta$ with $\phi(0) = 0$ must linearize $f_\alpha$ because: first $f(\Delta)\subset\Delta$, second the conjugate is a self-map of $r\D$ that fixes $0$ with multiplier of modulus one, so the case of equality of Schwarz's lemma implies that it is a rotation.

\begin{notation}[Linearizing map]\label{def:phi}
We let $\phi_\alpha : r(\alpha)\D \to \Delta(f_\alpha)$ be the unique uniformization such that $\phi_\alpha(0) = 0$ and $\phi'_\alpha(0) = 1$. 
\end{notation}

We also write
\[R_\alpha(z) = e^{2\pi i \alpha} z.\]
Now assume $\alpha$ is irrational. It is well known then that there is a unique formal power series $\Phi(X)$ with $\Phi(X) = X + \cal O(X^2)$ and $\Phi \circ R_\alpha = f_\alpha \circ \Phi$.
It follows that when $f_\alpha$ is linearizable, $\Phi$ is the power series expansion of $\phi_\alpha$. In particular the radius of convegence of $\Phi$ is greater or equal to $ r(\alpha)$.\footnote{They do not have to be equal, as the maps $f_\alpha$ may extend beyond $\D$ and the extension may well have a bigger Siegel disk.}

Consider any holomorphic map $\psi$ satisfying $\psi(0)=0$, $\psi'(0)=1$ and such that $\psi \circ R_\alpha = f_\alpha\circ \psi$ holds near $0$. Then $f$ is linearizable and if $\alpha\notin\Q$ then $\psi$ must coincide with $\phi_\alpha$ near $0$: this can be seen either by comparing to $\Phi$ as above, or more directly: $\phi_\alpha^{-1}\circ\psi$ has derivative $1$ at $0$, commutes with $R_\alpha$ and $\alpha$ is irrational so its power series expansion is reduced to a linear term only.

\begin{lemma}[Convergence of the linearizing maps]\label{lem:li}
Consider $\alpha_n\to \alpha$ and let $\rho=\liminf r(\alpha_n)$.
Then $r(\alpha)\geq \rho$.\footnote{It follows that $r(\alpha)\geq \limsup r(\alpha_n)$, see  \Cref{lem:usc}.}
Assume that $\rho>0$. If $\alpha\in\Q$ assume moreover that $\rho\geq r(\alpha)$, i.e.\ $\rho=r(\alpha)$. Then
$\phi_{\alpha_n} \tend \phi_\alpha $ uniformly on compact subsets of $\rho\D$.
\end{lemma}
\begin{proof}
If $\rho>0$ then the restrictions of $\phi_{\alpha_n}$ to $r_n\D$ with $r_n=\min(\rho,r(\alpha_n))\tend \rho$ take values in $\D$ thus form a normal family.
Consider any extracted limit $\phi:\rho\D\to \C$ of these restrictions. Since 
$\phi_{\alpha_n}(0)=0$ and $\phi_{\alpha_n}'(0)=1$ we have $\phi(0)=0$ and $\phi'(0)=1$.
Hence the limit is not constant and thus takes values in $\D$. Passing to the limit in $\phi_{\alpha_n}\circ R_{\alpha_n} = f_{\alpha_n}\circ\phi_{\alpha_n}$ we get that $\phi\circ R_\alpha = f_\alpha \circ \phi$ on $\rho\D$. 
It follows that $\phi$ takes values in $K(f_\alpha)$ and since it is open, in the interior of $K(f_\alpha)$ hence in $\Delta(f_\alpha)$. By Schwarz's inequality applied to $\phi_\alpha^{-1}\circ\phi$ we have $|\phi'(0)| \leq r(\alpha)/\rho$, i.e.\ $r(\alpha)\geq \rho$.
If $\rho\geq r(\alpha)$ then we are in the case of equality of
Schwarz's lemma, and hence $\phi_\alpha^{-1}\circ\phi$ is the identity on $\rho\D$.
Alternatively, if $\alpha\notin\Q$ then by the uniqueness argument before the statement of the present lemma, $\phi$ must be equal to $\phi_\alpha$ near $0$. By analytic continuation of equalities, they coincide on all of $\rho\D$.
\end{proof}

\begin{remark*}Note that the second claim sometimes fails if $\alpha\in\Q$ if we do not assume $\rho\geq r(\alpha)$: for instance to build a counterexample for $\alpha=0$ one may consider a vector field $dz/dt = \chi(z)$ which has a singularity at $0$ (i.e. $\chi(0)=0$) with eigenvalue $\chi'(0)=2\pi i$ but is not linear and let $f_t$ be the restriction to $\D$ of the time-$t$ map associated to this vector field. Then as $t\neq 0$ tends to $0$, the complement of the interior of $K(f_t)$  tends, in the sense of Hausdorff, to the complement of the interior of the set $K(\chi)$ where $K(\chi)$ denotes the set of points in $\D$ whose forward orbit by the vector field is defined for all times and never leaves $\D$. Moreover if $t$ is irrationnal, then $\Delta(f_t)$ is in fact independent of $t$ and equal to the component containing $0$ of the interior of $K(\chi)$. On the other hand $f_0 = \on{id}$ hence $K(f_0)=\D$.
\end{remark*}

Let us prove the first point in \Cref{prop:basic}:

\begin{lemma}[Upper semi-continuity of $r$]\label{lem:usc}
If $\alpha_n\tend \alpha$ then $r(\alpha) \geq\limsup r(\alpha_n)$.
\end{lemma}
\begin{proof}
Let $\rho = \limsup r(\alpha_n)$. For a subsequence $a_{n[k]}$ we have $r(\alpha_{n[k]})\tend \rho$. The claim then follows from the first conclusion of \Cref{lem:li} applied to the subsequence.
\end{proof}

\subsection{A remark on continuum theory}\label{subsec:comb}
 (This section can be skipped as it is not necessary in the rest of the article.)

A \emph{continuum} is a non-empty, compact and connected metrizable topological space. 
In continuum theory, there is an object called the \emph{Lelek fan}.
It is a universal object in the sense that any continuum with some specific set of properties (see \cite{Cha,AO}) is homeomorphic to the Lelek fan.
A variant is the following, called \emph{straight one sided hairy arc} in \cite{AO}, that they abbreviate \emph{sosha}, but we prefer to call it here a \emph{comb}. The Lelek fan can be recovered from this continuum by contracting the base to a point.
\begin{definition}[The comb]\label{def:comb}
A \emph{straight one sided hairy arc} is the sub-graph
\[C = \setof{(x,y)}{0\leq y\leq f(x)}\]
of a function $f:[0,1]\to[0,+\infty)$
such that:
\begin{itemize}
\item $f$ is upper semi-continuous,
\item $f$ is weakly\footnote{See \Cref{prop:basic} for a definition.} lower semi-continuous on the left and on the right,
\item both $\setof{x\in[0,1]}{f(x)>0}$ and $f^{-1}(0)$ are dense in $[0,1]$,
\item $f(0)=0=f(1)$.
\end{itemize}
Its \emph{base} is the segment $\setof{(x,0)}{x\in[0,1]}$.
\end{definition}
\noindent The first condition is equivalent to $C$ being closed. We call it a comb or \emph{the} comb to emphasize that it also possesses a form of topological uniqueness (see \cite{AO}).

Under the condition of the above definition, it was proved in \cite{AO}, Proposition~2.4, that  the image by $f$ of an interval $[a,b]\subset[0,1]$ with $a<b$  is of the form $[0,M]$ for some $M>0$. In particular: though $f$ is highly discontinuous, it  satisfies the intermediate value property.

Also, $C$ is the closure of the graph of $f$ (Corollary 2.5 in \cite{AO}). The fact that the closure is contained in $C$ follows from $f$ being upper semi-continuous and non-negative. The fact that this closure contains $C$ means that for any $(x,y)$ with $0\leq y\leq f(x)$ there exists a sequence $x_n\to x$ such that $f(x_n)\to y$. In fact if $x\neq 0$ or $1$ then there is such a sequence satisfying $x_n<x$ and there is another satisfying $x_n>x$.

From the Lelek fan people have derived topological models for the Julia set of some exponential maps \cite{AO}, and conjecturally for the hedgehogs associated to non-linearizable fixed points of some polynomials. We will also see here a comb, though not as a subset of some dynamical plane, but as the subgraph of the function $\alpha\mapsto r(\alpha)$ in the special case of \Cref{sub:assopt}.

\subsection{Special case: assuming Brjuno's condition is optimal}\label{sub:assopt}

Assume here that we have a family for which we know that the Brjuno condition is optimal, in the sense that $r(\alpha)>0 \implies \alpha\in\cal B$ where $\cal B$ denotes the set of Brjuno numbers. The first family for which optimality has been known is the family of degree two polynomials, thanks to the work of Yoccoz, see \cite{Y}.

\begin{lemma}\label{lem:aa}
For all $\alpha\in I$ and all $y$ with $0\leq y< r(\alpha)$, the set $r^{-1}(y)$ (is non-empty and) accumulates $\alpha$ on the left and on the right. In other words there exists $\alpha_n<\alpha<\alpha'_n$ with $\alpha_n\to \alpha$ and $\alpha'_n\to \alpha$ and such that $r(\alpha_n)=y=r(\alpha'_n)$.
\end{lemma}
\begin{proof} (from \cite{A})
Since $\R\setminus\cal B$ is dense, there is a dense set on which $r=0$. In particular the case $y=0$ is trivial. Assume $y>0$.
Arbitrarily close to $\alpha$, there are $b\in I$ such that $r(b)=0$. Assume $b<\alpha$.
Consider then $K=\setof{x\in[b,\alpha]}{r(x)\geq y}$, which is non-empty because $\alpha\in K$, and let $c=\inf K$. By upper semi-continuity $r(c)\geq y$.
In particular $r(c)\neq 0$ hence $c\neq b$ and $c\in \cal B$ by optimality assumption.
If we had $r(c)>y$ then by weak lower semi-continuity on the left at Brjuno numbers (\Cref{prop:basic}), there would be some $c'\in(b,c)$ with $r(c')>y$, contradicting the definition of $c$. The same holds if $b>\alpha$ using weak lower semi-continuity on the right. Finally $r(c)=y \neq r(\alpha)$ ensures that $c\neq \alpha$.
\end{proof}

\begin{remark*}
Consider an interval $[u,v]\subset I$ with $u<v$ and $r(u)=0=r(v)$ (there are plenty by hypothesis), and let $C$ be the subgraph of $r$ restricted to $[u,v]$. Then $C$ is a comb as per \Cref{def:comb}: we imposed $r(u)=0=r(v)$ and all remaining conditions are satisfied according to \Cref{prop:basic}. This implies the lemma above, by the discussion in \Cref{subsec:comb}, and in fact the proof of the intermediate value property for a general comb boils down to the same arguments as the proof of \Cref{lem:aa}.
\end{remark*}

Since we assumed optimality of Brjuno's condition, the values $\alpha_n$ in \Cref{lem:aa} belong to $\cal B$. By \Cref{compl:bdd} we get:

\begin{corollary}\label{cor:bdd}
For all $\alpha\in I$ and all $y$ with $0\leq y\leq r(\alpha)$, there exists $\alpha_n<\alpha<\alpha'_n$ with $\alpha_n\to \alpha$ and $\alpha'_n\to \alpha$ and such that: $\alpha_n$ and $\alpha'_n$ are bounded type irrationals and $r(\alpha_n)$ and $r(\alpha'_n)$ both tend to $y$.
\end{corollary}

\noindent In other words we gain information on the arithmetic type of $\alpha_n$ at the cost of weakening ``$r(\alpha_n)=y$'' into ``$r(\alpha_n)\to y$''. Note also that we gain the ability to reach $y=r(\alpha)$.

\subsection{Smooth Siegel disks}

Recall the standing assumption that the family depends continuously on $\alpha$. In \cite{A, ABC} is shown how to get smooth Siegel disks from \Cref{lem:aa}, which assumes Brjuno's condition is optimal.
We adapt here the proof so that it does not use optimality directly but only depends on the following condition:\footnote{Other conditions are sufficient to apply the methods of \cite{BC}. For instance we can replace the assumption $\alpha\in \cal B $ by $\alpha$ having bounded type. Also, bounded type numbers in the hypothesis and conclusion can be replaced by Herman numbers.}

\begin{condition}\label{cond:bdd}
For every $\alpha\in\cal B$ and every $\rho\in\R$ with $0<\rho<r(\alpha)$, there exists a sequence of bounded type numbers $\alpha_n \tend \alpha$ such that $r(\alpha_n)\tend \rho$.
\end{condition}

\smallskip

For a family on which the Brjuno condition is optimal, \Cref{cor:bdd} implies that \Cref{cond:bdd} holds. Better: it provides a sequence on \emph{each} side, though the construction below does not need it.
%\begin{remark*}
%The proof of \Cref{cor:bdd} makes a clever use of weak lower semi-continuity of $f$ on the left and on the right.
%If we were to prove, for the general case, a weakened version of \Cref{cor:bdd} where we would \emph{not} specify a side from which $\alpha_n$ tends to $\alpha$, we believe it would is still be crucial for such a proof to know that $r$ is weakly lower semi-continuous either everywhere on the left or everywhere on the right.
%\end{remark*}

Recall that the point of this article is to get rid of the hypothesis that the Brjuno condition is optimal: we  will prove in \Cref{cor:follow} that \Cref{cond:bdd} holds whenever the family is Lipschitz-continuous with respect to the parameter, and non-degenerate as per \Cref{def:dege}.

\medskip

We now begin the construction of a smooth Siegel disk assuming \Cref{cond:bdd}. Recall that $\phi_\alpha$ denotes the unique conformal bijection from $r(\alpha)\D \to \Delta(f_\alpha)$ such that $\phi_\alpha(z) = z +\cal O(z^2)$ and that $\phi_\alpha$ linearizes $f_\alpha$.
%, and that optimality means $r(\alpha)>0$ $\implies$ $\alpha\in \cal B$ where $\cal B$ is the set of Brjuno numbers.

\medskip
\noindent{\it Construction of a sequence $\theta_n$.}

Start from any $\theta_0\in\cal B$, so that $r(\theta_0)>0$ and also $\theta_0\notin\Q$. Choose some target radius $\rho \in(0,r(\theta_0))$. Choose also a strictly decreasing sequence $\rho_n$ such that $\rho_0=r(\theta_0)$ and $\rho_n \tend \rho$.

Consider then a sequence $\alpha_k\to \theta_0$, provided by \Cref{cor:bdd}, such that $r(\alpha_k) \tend \rho_1$ and $\alpha_k\in\cal B$.
From the properties of linearizing maps (\Cref{lem:li}) and $\theta_0\notin\Q$ it follows that $\phi_{\alpha_k} \to \phi_{\theta_0}$ uniformly on every compact subset of $\rho_1\D$. Since these are holomorphic maps, the same is true for their derivatives of all orders. We let $\theta_1 =\alpha_k$ for a choice of $k$ such that the restriction of $\phi_{\theta_1}-\phi_{\theta_0}$ to the closure of $\rho\D$ is less than $1/2$ (for the sup norm) and such that its first derivative is less than $1/4$.
Since $\alpha_k\in\cal B$ it follows that $\theta_1\notin\Q$.

Then we choose some open interval $I_1$ of length $\leq 1/2$, containing $\theta_1$, and such that $\alpha\in \ov I_1$ $\implies$ $r(\alpha)\leq \rho_0$ (upper semi-continuity of $r$ at $\theta_1$).
We also ask that the the closure of $I_1$ lies at positive distance from $\Z$, which is possible since $\theta_1\notin\Z$.

We then continue the inductive construction: given $n\geq 2$, once $\theta_{n-1}\in \cal B$ and $I_{n-1}$ with $\theta_{n-1} \in I_{n-1}$ have been constructed we choose $\theta_n$ of the form $\alpha_k$ where $\alpha_k\in \cal B$ is a sequence provided by \Cref{cor:bdd} tending to $\theta_{n-1}$ such that $r(\alpha_k) \tend \rho_n$.
The index $k$ is chosen big enough so that: the restriction to the closure of $\rho \D$ of $\phi_{\theta_n}-\phi_{\theta_{n-1}}$ is less than $1/2^n$, its first derivative less than $1/2^{n+1}$ and so on up to its $n$-th derivative less than $1/2^{n+n}$.
It is also chosen big enough so that $\theta_n$ belongs to the interior of $I_{n-1}$, a condition that we did not have for $n=1$.

Then we choose an open sub-interval $I_n\subset I_{n-1}$ of lenght $\leq 1/2^n$, containing $\theta_n$, such that $\alpha\in\ov I_n$ $\implies$ $r(\alpha)\leq \rho_n$ and such that the closure of $I_n$ lies at positive distance from $\frac{1}{n}\Z$.
And so on\ldots

\medskip
\noindent{\it Properties of the limit $\theta$.}

The intersection of the closures of the $I_n$ is a singleton and $\theta_n$ tends to this value.
Since $r(\theta_n)=\rho_n\tend \rho$, upper semi-continuity of $r$ implies $r(\theta)\geq \rho$.
By the defining properties of $I_n$, which contains $\theta$, we get $r(\theta)\leq \rho_n$ for all $n$ so $r(\theta)\leq \rho$.
Hence $r(\theta)=\rho$.
Also $\theta$ belongs to the closure of $I_n$ for all $n$, so the distance from $\theta$ to $\frac{1}{n}\Z$ is positive for all $n>0$ hence $\theta\notin\Q$.
The conditions on $\phi_{\theta_n}-\phi_{\theta_{n-1}}$ implies the uniform convergence on the closure of $\rho\D$ of the derivatives of all orders of $\phi_{\theta_n}$.
Since $\theta\notin\Q$, \Cref{lem:li} implies that $\phi_{\theta_n}\to \phi_\theta$ uniformly on compact subsets of $\rho\D$.\footnote{However even if we had $\theta\in\Q$, since $r(\theta_n)\to r(\theta)$, we would still have $\phi_{\theta_n}\to \phi_\theta$.} It follows that $\phi_{\theta}$ has a $C^{\infty}$ extension $\tilde\phi$ to the closure of $\rho\D$.
The image of the circle $\rho\partial\D$ by this extension is the boundary of $\Delta(f_\theta)$.
By a straightforward modification of the construction above, we can ensure that the curve is compactly contained in $\D$. Then $\tilde\phi \circ R_{\theta} = f_\theta\circ \tilde \phi$ also holds on the boundary circle.
Hence the derivative of $\tilde\phi$ on cannot vanish on this circle, for if it did, it would vanish on a dense subset  using the equation above, hence everywhere on the circle, whence $\tilde\phi$ would be constant by standard properties of holomorphic functions. The map $\tilde \phi$ is also injective on the boundary circle (see \cite{M}, Lemma~18.7 page 193; it is stated for rational maps but is valid as soon as the rotation number is irrational and the Siegel disk compactly contained in the domain of the map).

Hence $\partial \Delta(f_\theta)$ is a $C^\infty$ Jordan curve compactly contained in $\D$.

\begin{remark*} An interesting feature of this construction is that we were able to prescribe the conformal radius of the Siegel disk.\footnote{By a linear change of variable $z\mapsto \lambda(\alpha) z$ with $\lambda$ continuous, we can thus prescribe the conformal radius to coincide with a continous function of $\alpha$.}
\end{remark*}

\subsection{Other regularity classes for the boundary}\label{sub:other}

The construction in \cite{BC} of boundaries that are $C^n$ but not $C^{n+1}$, and of other examples (see \Cref{app:grl}), is a refinement of the previous method. In the process we loose the ability to exactly prescribe the conformal radius.

To apply the method of \cite{BC} and thus get \Cref{thm:main}, it is enough to have a continuous family of maps $f_\alpha$ such that \Cref{cond:bdd} holds.
In \cite{BC} there are two supplementary condition, but we can remove them:
\begin{itemize}
\item The maps $f_{\alpha}$ must be injective. But if the given family $f_\alpha$ contains non-univalent maps, we first restrict to a sub-interval $I$ and restrict $f$ to a small enough disk $\eps \D$. The Siegel disk $\Delta'$ that it produces for the restriction to $\eps\D$ of $f$ is compactly contained in $\eps\D$ and thus $\Delta(f) = \Delta'$ by the end of \Cref{cor:subdisk}.
\item The family must depend analytically on $\alpha$. But it turns out that the proof given in \cite{BC} only uses continuity of the family and the fact that \Cref{cond:bdd} holds.
\end{itemize}

\begin{remark*}
The proof that \Cref{cond:bdd} is enough is a bit elaborate so we refer the reader to \cite{BC}.  Let us just mention that in the construction, to get obstructions to regularity we use as intermediate steps the existence of Siegel disks whose boundaries oscillate a lot. 
For this we use a theorem of Herman \cite{H}: if $\alpha$ is a Herman number (this includes all bounded type numbers),\footnote{Notably, Yoccoz completely determined the set of Herman numbers $\alpha$ in terms of simple manipulations of $\alpha$, see \cite{Y2}.} and if $f$ is univalent then $\Delta(f_\alpha)$ cannot be compactly contained in $\D$. Now if we have a sequence $\alpha_n\to\alpha$ of Herman numbers such that $r(\alpha_n)\to \rho \in (0,r(\alpha))$ then the sets $\partial \Delta(f_{\alpha_n})$ have a point in $\partial \D$ but also points close to the $f_\alpha$-invariant curve $\phi_\alpha(\rho\partial\D) \subset \Delta(\alpha)$. See \cite{BC} for the rest of the argument.
As an alternative to Herman's theorem, we can use the following consequence of \cite{GS} (whose methods are quite different from \cite{H}): If $\alpha$ has bounded type then $\partial \Delta$ either meets $\partial \D$ or contains a critical point of $f$. Note that it is known only for bounded type rotation numbers.
\end{remark*}

\subsection{General case}

Here we prove that \Cref{cond:bdd} holds for families that are non-degenerate (in the sense of \Cref{def:dege}) and for which the dependence on $\alpha$ is Lipschitz continuous. This yields \Cref{thm:main}.

This is based on the following perturbation lemma, to the proof of which we devote the whole of \Cref{sec:proof}.
\begin{lemma}[Main lemma: perturbation of a rotation]\label{lem:main}
Let $\alpha\in \R$ and let $[a_0;a_1,\ldots]$ be its continued fraction expansion.
If $\alpha\in \Q$ then it has two such expansions\footnote{See the end of \Cref{sub:remind}.} and both are finite : choose one, $[a_0;a_1,\ldots,a_k]$.
\begin{enumerate}
\item If $\alpha\notin\Q$ let $\alpha_n = [a_0;a_1,\ldots,a_n,1+a_{n+1},1+\sqrt{2}]$.
\item If $\alpha\in\Q$ let $\alpha_n = [a_0;a_1,\ldots,a_k,n+1+\sqrt{2}]$.
\end{enumerate}
Assume that $f_n$ are holomorphic functions defined on $\D$ with $f_n(z)=e^{2\pi i\alpha_n} z +\cal O(z^2)$ and that
$f_n$ tends to $R_\alpha$ in a Lipschitz way with respect to $\alpha_n-\alpha$, i.e.:
\[|f_n (z)- R_{\alpha}(z)|\leq K|\alpha_n-\alpha|\]
for some $K\geq 1$.\footnote{We necessarily have $K\geq 1$ by Schwarz's lemma comparing derivatives at the origin. To allow for smaller values of $K$ we would compare $f_n$ to $R_{\alpha_n}$ instead of $R_{\alpha}$. However we are not interested in small values of $K$ in this article.}
Then
\begin{enumerate}
\item If $\alpha\notin\Q$ then \[\liminf r(f_n) \geq 1.\]
\item If $\alpha = p/q$ in irreducible form, then
\[\liminf r(f_n) \geq \exp(-C(K,q)).\]
\end{enumerate}
Here $C(K,q)>0$ is independent of $\alpha$, of the sequence $f_n$ and of $\alpha_n$ and satisfies the following: for all integers $q\geq 1$, $K\mapsto C(K,q)$ is a continuous non-decreasing function of $K\geq 1$; for all fixed $K\geq 1$ we have $C(K,q)\tend 0$ as $q\to+\infty$.
\end{lemma}

\begin{remark}
If $\alpha\in\cal B$ then \Cref{lem:main} is already known: it follows from \cite{BC} or \cite{R}. So the novelty is for non-Brjuno numbers and rational numbers.
\end{remark}

In \Cref{sec:proof} we prove that the following value of $C(K,q)$ works: 
\begin{equation}\label{eq:C}
C(K,q) = \frac{\log q}{q} + \frac{\log K}{q} +\frac{c_1}{q}
\end{equation}
for some $c_1>0$. This estimate may be non optimal.

Note that $\alpha_n\to\alpha$. If $\alpha\notin\Q$ then $\alpha_{2n}<\alpha<\alpha_{2n+1}$.
If $\alpha\in\Q$ then $\alpha_n$ is on one side of $\alpha$ or the other depending on which of the two continued fraction was chosen.

Of course the choice of $1+\sqrt2$ is somewhat arbitraty and many other variants hold.
Recall that $1+\sqrt2 = [2; 2,2,2,\ldots]$, so it is a close relative to the golden mean $\frac{1+\sqrt5}{2} = [1;1,1,1,\ldots]$, that we can use instead. Note that the class of $\alpha_n \bmod \Z$ only depends on the class of $\alpha \bmod \Z$ (and on $n$).

The fact that we do not get $\liminf r(f_n) = 1$ when $\alpha\in\Q$ is not just a limitation of our method. Indeed, as in the remark following \Cref{lem:li}, consider a vector field $\dot z = \chi(z)$ defined in a neighborhood of the closed unit disk and with $\chi(z)=2\pi i z + \cal O(z^2)$ and let $h_t$ be the associated time-$t$ map. If the vector field is invariant by the rotation $R_{1/q}$ we may set $f_t=R_{p/q}\circ h_{t-p/q}$. Then the family satisfies $\sup_{\D}|f_t-R_{p/q}| \sim_{t\to p/q} K |t-p/q| $ with $K=\sup_\D |\chi|$. For $t$ irrational, its Siegel disk is independent of $t$ and coincides with the maximal linearization domain for $\chi$, which is not the whole unit disk, except if $\chi(z) = 2\pi i z$ for all $z$.

%\arnaud{On pourrait faire un commentaire sur ce que cela implique sur l'optimalité de \cref{eq:C}.}

\medskip

As a consequence of the main lemma we now prove:

\begin{lemma}[Perturbation lemma for Lipschitz families]\label{lem:pert}
Let $I$ be a non-empty interval and $(f_\alpha)_{\alpha\in I}$ a family of functions $\D\to\C$ with expansion 
$f_\alpha(z)=e^{2\pi i\alpha} z + \cal O(z^2)$ at $0$. Assume that the family is $K$-Lipschitz for some $K\geq 1$, i.e.\ $\forall \alpha,\beta \in I$ and $\forall z\in\D$,
\[|f_\alpha(z)-f_\beta(z)|\leq K |\alpha-\beta|.\]
Then for all $\alpha\in I$, if we write $\alpha_n$ an associated sequence like in \Cref{lem:main}, we have
\begin{enumerate}
\item If $\alpha\notin\Q$ then \[\liminf r(f_{\alpha_n}) \geq r(f_\alpha).\]
\item If $\alpha = p/q$ in irreducible form, then
\[\liminf r(f_{\alpha_n}) \geq r(f_\alpha)/\exp(C'(K,q)).\]
\end{enumerate}
Similarly to the previous lemma, $C'(K,q)$ is independent of the family $(f_\alpha)$, for each $q$ the function $K\mapsto C'(K,q)$ is continuous, non-decreasing and for each $K\geq 1$, $C'(K,q)\tend 0$ as $q\to+\infty$.
\end{lemma}
Here we can get: 
\begin{equation}\label{eq:C2}
C'(K,q) = 4\frac{\log q}{q} + \frac{\log K}{q} + \frac{c_2}{q}
\end{equation}
where $c_2$ is a positive universal constant. As in \cref{eq:C}, this estimate may be non-optimal.
\begin{proof}
If $r(\alpha)=0$ the claim is trivial so we assume $r(\alpha)>0$.

First, we can immediately improve the inequality $|f_\alpha(z)-f_\beta(z)|\leq K |\alpha-\beta|$ by Schwarz's inequality because $f_\alpha-f_\beta$ maps $0$ to $0$:
\[|f_\alpha(z)-f_\beta(z)|\leq K |\alpha-\beta| \cdot|z|.\]
Consider the linearizing map $\phi_\alpha$.
Let $g_\beta = \phi_\alpha^{-1} \circ f_\beta \circ \phi_\alpha$.
Then $g_\alpha$ is the restriction of $R_\alpha$ to $r(\alpha)\D$. 
The maps $g_\beta$ are defined on subsets $\dom g_\beta$ of the disk $r(\alpha)\D$ that tend to this disk in the following sense: $\forall r<r(\alpha)$, $\exists \eta>0$ such that $|\beta-\alpha|<\eta$ $\implies$ $r\D\subset \dom g_\beta$. 

Fix for a moment a value $r<r(\alpha)$. Write $\eps = 1- r/r(\alpha)$ so that $r=(1-\eps)r(\alpha)$. Consider the family
\[\tilde f_\beta = r^{-1}g_\beta(r z)\]
restricted to $\D$ and to values $\beta$ such that $|\beta-\alpha|<\eta$ where $\eta$ is as above, so that $\tilde f_\beta$ is indeed defined on the whole of $\D$.
We show that for $\beta$ close enough to $\alpha$ the family $\tilde f_\beta$ is $K'$-Lipschitz for a constant $K'$ that we determine.

We will use the following two property of univalent maps, see \cite{Po}, Theorem~1.3 page~9: if $\phi:\D\to \C$ is holomorphic, injective and satisfies $\phi(z)=z + \cal O(z^2)$ at $0$ then $|\phi(z)|\leq |z|/(1-|z|)^2$ and $|\phi'(z)|\geq (1-|z|)/(1+|z|)^3$.
These bound are optimal because the Koebe function $f(z) = z/(1-z)^2$ reaches them.

Transferring them to the function $\phi_\alpha$ by letting $\phi(z)=r(\alpha)^{-1}\phi_\alpha(r(\alpha)z)$ this implies
\[\forall z \in r\D,\quad |\phi_\alpha(z)| \leq \frac{|z|}{(1-\frac{r}{r(\alpha)})^2}\qquad
|\phi'_\alpha(z)| \geq \frac{1-\frac{r}{r(\alpha)}}{(1+\frac{r}{r(\alpha)})^3},\]
which implies
\[\forall z \in r\D,\quad |\phi_\alpha(z)| \leq \frac{|z|}{\eps^2}\qquad
|\phi'_\alpha(z)| \geq \frac{\eps}{8}.\]

From the lower bound on $\phi'_\alpha$ it follows that the family $g_\beta$ satisfies the following estimate:
\[\forall z\in r\D,\quad |g_\beta(z)-g_\alpha(z)| \leq K(\beta) |\beta-\alpha|\cdot |\phi_\alpha(z)|\]
with $K(\beta)\tend 8K/\eps $ as $\beta\to\alpha$:
this can be proved for instance by contradiction and extracted subsequences for $|z|\in[r/2,r]$ and then by the maximum principle it extends to $|z|< r$.

Transferring to $\tilde f_\beta$ and using the upper bound on $\phi_\alpha(z)$ we get
\[ \forall z\in r\D,\quad
|\tilde f_\beta(z) - \tilde f_\alpha(z)| \leq K'(\beta) |\beta-\alpha|\cdot |z|
\]
with $K'(\beta) \tend 8K/\eps^3$ so we can take a uniform $K' = 9K/\eps^3$ by requiring $\beta$ to be close enough to $\alpha$.

We now apply \Cref{lem:main} to the family $\tilde f_\beta$.

If $\alpha\notin\Q$ we get that $\liminf r(\tilde f_{\alpha_n}) \geq 1$ hence $\liminf r(f_{\alpha_n}) \geq r$. Since this is true for all $r<r(\alpha)$ this implies $\liminf r(f_{\alpha_n})\geq r(f_\alpha)$.

If $\alpha\in\Q$ we get $\liminf r(\tilde f_{\alpha_n}) \geq \exp(-C(K',q))$ with $K'=9K/\eps^3$. Recall that $r=(1-\eps)r(\alpha)$ hence $\liminf r({\alpha_n}) \geq r\exp(-C(K',q)) = r(\alpha) \exp(-Q)$ with
\[Q = -\log(1-\eps)+C(9K/\eps^3,q).\] Since this is true for all $\eps\in(0,1)$ we get
$\liminf r({\alpha_n}) \geq r(\alpha)\exp(-C'(K,q))$ with 
\[C'(K,q) := \inf\setof{-\log(1-\eps)+C(9K/\eps^3,q)}{\eps \in (0,1) }.\]

The map $K\mapsto C'(K,q)$ is continuous. One argument to prove this claim goes as follows: by \Cref{lem:main} the map $K\mapsto C(K,q)$ is continuous and has a limit as $K\to+\infty$ because it is monotonous. Hence the expression $Q$ extends to a continuous function of $(\eps,K)$ from $[0,1]\times (0,+\infty)$ to $[0,+\infty]$ where the topology is extended to include infinity in the range. This is a sufficient condition for the function $K\mapsto \inf_{\eps\in(0,1)} Q(\eps,K,q)$ to be continuous.

Increasing $K$ while fixing $q$ and $\eps$ does not decrease $Q$ hence $K\mapsto C'(K,q)$ is non-decreasing. 

For each $q$, fixing $K$ and $\eps$, we have $Q\tend -\log(1-\eps)$ when $q\to +\infty$ and this quantity can be made close to $0$ by choosing $\eps$ small. Hence $C'(K,q)\tend 0$ when $K$ is fixed and $q\to+\infty$.
\end{proof}

\begin{proof}[Proof of \Cref{eq:C2} from \cref{eq:C}] Using the notation of the proof above,
we must derive an upper bound for the infimum over $\eps\in(0,1)$ of $-\log(1-\eps) + c_1/q+ \frac{1}{q} \log (9Kq/\eps^3)$.
This is a function of $\eps$ whose derivative has the following simple expression $\frac{1}{1-\eps} -\frac{3}{q\eps}$. So the function is strictly convex with infinite limits at $\eps=0$ and $\eps=1$, and a unique minimum at $\eps = 1/(1+q/3)$. We get
$C'(K,a) = \log(1+3/q) + c_1/q+  \log (9Kq)/q + 3 \log(1+q/3)/q \leq \log(K)/q +4\log(q)/q+ c_2/q$.
\end{proof}

\begin{corollary}\label{cor:follow}
If the family $(f_\alpha)$ is non-degenerate in the sense of \Cref{def:dege} and Lipschitz-continuous with respect to $\alpha$ then 
\Cref{cond:bdd} holds.
\end{corollary}

Recall that \Cref{cond:bdd} is stated as follows: For every Brjuno number $\alpha$ and $\rho\in\R$ with $0<\rho<r(\alpha)$, there exists a sequence of bounded type numbers $\alpha_n \tend \alpha$ such that $r(\alpha_n)\tend \rho$. Here we moreover get that there is such a sequence on each side of $\alpha$.

\begin{proof}
We adapt the proof of \Cref{lem:aa}.

Let $\alpha\in\cal B$ and $\rho \in \R$ with $0<\rho<r(\alpha)$. By the non-degeneracy assumption, arbitrarily close to $\alpha$ there are $b\in I$ such that $r(b)=0$. Choose one and assume for simplicity that $b<\alpha$ (the other case is similar). Consider then $K=\setof{x\in[b,\alpha]}{r(x)\geq \rho}$, which is non-empty because $\alpha\in K$, and let $c=\inf K$. By upper semi-continuity
\[r(c)\geq \rho.\]
In particular $r(c)\neq 0$, hence $c\neq b$.

Here, compared to \Cref{lem:aa}, we cannot anymore deduce that  $c\in \cal B$ because we do not assume optimality.
Instead, we use \Cref{lem:pert} with $c$ in place of $\alpha$. It provides some special sequence $\alpha_n\tend c$ of bounded type numbers. If $c\notin \Q$ we let $c_n = \alpha_{2n}<c$. If $c\in \Q$ the sequence $\alpha_n$ is either below or above $c$, depending which continued fraction of $c$ was chosen among the two possible, so we choose it so that $\alpha_n<c$ and let $c_n=\alpha_n$. For all $n$ big enough we have $c_n\in(b,c)$. Then by definition of $c$, we have $r({c_n})<\rho$.
Now there are two cases. 
\begin{itemize}
\item Either $c\notin \Q$. Then \Cref{lem:pert} states that $\liminf r({c_n})\geq r(c)$ so
\[\lim r(c_n) = \rho.\]
We can thus choose $n$ so that $r(c_n)$ is arbitrarily close to $\rho$.
\item Or $c\in\Q$. Then \Cref{lem:pert} states that $\liminf r({c_n})\geq r(c)/\exp(C'(K,q))$ where $q$ is the denominator of $c=p/q$ written in irreducible form. So
\[\liminf r({c_n}) \in [\rho,\rho/\exp(C'(K,q))].\]
\end{itemize}
The number $c\in(b,\alpha)$ above depends on the choice of $b$. If we now let $b$ tend to $\alpha$ then $c$ tends to $\alpha$ and in particular: whenever $c$ is rational its denominator $q$ tends to $+\infty$, so $C'(K,q)$ tends to $0$. 
%In all cases $r(c)$ tends to $\rho$. From $\rho\neq r(\alpha)$ we deduce $c\neq\alpha$ when $b$ is close enough to $\alpha$.
\end{proof}

As mentionned in \Cref{sub:other}, \Cref{cond:bdd} is all that is needed to get the results of \cite{BC}.
Hence by the corollary above the results of \cite{BC} extend to all families that are Lipschitz and non-degenerate. In particular we have \Cref{thm:main}.

\section{Proof of the main lemma}\label{sec:proof}

We will use a construction due to Douady and Ghys
which has been quantified by Yoccoz (see \cite{D}, \cite{Y} or
\cite{PM2}), and is called \emph{sector renormalization}.
It has been treated in many articles and books since, so we will not
motivate its construction here.

\subsection{Lifts}\label{sub:lift}

\subsubsection{Definitions}\label{subsub:defs}

Let
\begin{itemize}
\item $\H$ denote the upper half plane,
\item $T: Z\mapsto Z+1$.
\item $T_\alpha: Z\mapsto Z+\alpha$.
\item For $\alpha\in \R$, let $\cal S(\alpha)$ be the space of univalent (i.e.\ injective holomorphic) maps $F:\H\to \C$ such that $F\circ T=T\circ F$ holds on $\H$ and such that $F(Z)-Z\tend \alpha$ as $\im(Z)\to +\infty$.
\end{itemize}

We call $\alpha$ the \emph{rotation number} of $F$ even though it is rather a translation that $F$ is compared to, and we write it $\alpha(F)$.
A map $F\in\cal S(\alpha)$ satisfies the property:
\[F(Z)=Z+\alpha+h(e^{2\pi i z})\]
for a holomorphic map $h:\D\to \C$ with $h(0)=0$.

\medskip

The map
\[E(z) = e^{2\pi i z}
\]
is a universal cover from $\C$ to $\C^*=\C\setminus\{0\}$
and its restriction to $\H$ is a universal cover from $\H$ to $\D^* = \D\setminus\{0\}$.

For any univalent map $f:\D\to \C$ which fixes $0$ with derivative $e^{2i\pi \alpha}$, 
a \emph{lift} is a holomorphic map $F$ such that
$f\circ E = E\circ F$. Then $F\in \cal S(\alpha')$ for some $\alpha'\equiv \alpha\bmod\Z$.
Lifts exist and are unique if we require $\alpha'=\alpha$. Conversely every $F\in \cal S(\alpha)$ arises as the lift of a (unique) univalent map $f$ as above.

Given $\alpha\in \R$, the space ${\cal S}(\alpha)$ is compact for the
topology of uniform convergence on compact subsets of $\H$. If $F\in
{\cal S}_\alpha$, we let $K(F)$ be the set of points $Z\in \H$ whose
orbit under iteration of $F$ remains in $\H$. For the corresponding $f$, we have
\[K(f) = \{0\}\cup E(K(F))\] and $\Delta(f) = \emptyset$ if $\Delta(F) =\emptyset$, otherwise $\Delta(f) = \{0\}\cup E(\Delta(f))$.
We say that $F$ is linearizable whenever the corresponding map $f$ is linearizable. This happens if and only if $K(F)$ contains an
upper-half plane and we write
\[h(F) =\inf\setof{h>0}{K(F)\text{ contains }``\Im Z>h"}.\]
We set $h(F)=+\infty$ otherwise.
We then have
\[r(f) \geq e^{-2\pi h(F)}.\]

\subsubsection{Transfer of the Lipschitz condition to the lifts}

Consider the lifts $T_\alpha$, $F_n\in \cal S(\alpha_n)$ of $R_\alpha$, $f_n$.
We can factor $f_n$ as follows: $f_n(z)=R_\alpha(z) g_n(z)$ with $g_n(0)=e^{2\pi i(\alpha_n-\alpha)}\neq 0$.
Then from $|f_n(z)-R_\alpha(z)|\leq K|\alpha_n-\alpha|$ on $\D$ we get $|g_n(z)-1|\leq K|\alpha_n-\alpha|$ by a form of the maximum principle. In particular for $n$ big enough we have that $\|g_n-1\|_\infty \leq 1/2$. 
Then $F_n(Z) - T_{\alpha}(Z) = \log g_n(E(z))$ for the principal branch of $\log$. (For this we also have to take $n$ big enough so that $|\alpha_n-\alpha|<1$, but note that with the special sequence $\alpha_n$ under consideration, it already holds for all $n\geq 0$.) Since the derivative of $\log$ has modulus less than $2$ on $B(1,1/2)$ we get
\begin{eqnarray*}
&&\forall Z\in\H,\ |F_n(Z) -Z-\alpha| \leq 2K|\alpha_n-\alpha|,\\
&&\forall Z\in\H,\ |F_n(Z) -Z-\alpha_n| \leq (2K+1)|\alpha_n-\alpha|.
\end{eqnarray*}

In the rest of \Cref{sec:proof} we will prove the following version of the main lemma (\Cref{lem:main}):

\begin{lemma}\label{lem:main2}
There exists a continuous function $C(K)$ such that for all $\alpha$, if we define $\alpha_n$ as in \Cref{lem:main} (we repeat the definition below for convenience) then for all $K\geq 1$ and all sequence $F_n\in \cal S(\alpha_n)$ such that 
\[|F_n (Z)- Z-\alpha_n|\leq K|\alpha_n-\alpha|:\]
\begin{enumerate}
\item If $\alpha$ is irrationnal then
\[\limsup h(F_n) \leq 0.\]
\item If $\alpha = p/q$ in irreducible form, then
\[\limsup h(F_n) \leq C''(K,q)= \frac{\log (Kq)}{2\pi q} + \frac{c_3}{q}\]
for some universal constant $c_3>0$.
\end{enumerate}
\end{lemma}

\noindent It implies \Cref{lem:main}, with the constant $C(K,q) = 2\pi C''(2K+1,q) \leq \log (Kq)/q + c_1/q$ for some universal constant $c_1>0$. 

\medskip

\begin{reminder} For convenience, we repeat here the definition of $\alpha_n$ given in \Cref{lem:main}: let $[a_0;a_1,\ldots]$ be the continued fraction expansion of $\alpha$.
If $\alpha\in \Q$ then it has two such expansions\footnote{See the end of \Cref{sub:remind}.} and both are finite : we choose one, $[a_0;a_1,\ldots,a_k]$.
\begin{enumerate}
\item If $\alpha\notin\Q$ we let $\alpha_n = [a_0;a_1,\ldots,a_n,1+a_{n+1},1+\sqrt{2}]$.
\item If $\alpha\in\Q$ we let $\alpha_n = [a_0;a_1,\ldots,a_k,n+1+\sqrt{2}]$.
\end{enumerate}
\end{reminder}

\medskip

The proof of \Cref{lem:main2} is based on renormalization and on the following:
\begin{lemma}\label{lem:s2}
There exists $C_{\sqrt2}>0$ such that $\forall F\in \cal S(\sqrt2)$, $h(F)\leq C_{\sqrt 2}$.
\end{lemma}
\noindent It is a consequence of Brjuno's or Siegel's theorems, and can also be proved using renormalization (Yoccoz), see \cite{Br,Ru,S,Y}.

\medskip

Note that for the proof of \Cref{lem:main2} it is enough to assume $\alpha\in[0,1)$ because $h(T^{-a}\circ F) = h(F)$ for any $a\in\Z$,
and shifting $\alpha$ by an integer shifts by the same amount the special sequences $\alpha_n$ defined in \Cref{lem:main}.

\subsection{Gluing}\label{sub:glue}

Rernormalization uses gluings, described below. We present here first a simplified version and an associated basic estimate. The next section will transpose this construction to a class of lifts.

Let $\ell=i(0,+\infty)$, i.e. half of the imaginary axis, endpoint excluded.
Consider a holomorphic map $F$ defined in a neighborhood\footnote{Since $\ell$ does not contain its endpoint, it means that the inner radius of such a neighborhood $V$ may shrink near this point. Soon we will consider the case $V=\H$.} of $\ell$ and such that:
\begin{equation}\label{eq:one}
(\forall W\in\ell) \quad |F(W) - W - 1| \leq 1/10 \quad\text{and}\quad |F'(W)-1| \leq 1/10.
\end{equation}
We do \emph{not} assume here that $F$ belongs to some $\cal S(\alpha)$.
The curve $\ell\cup [0,F(0)]\cup F(\ell)$ bounds an open strip $U$ in $\C$. See \Cref{fig:L}.
Gluing the boundaries $\ell$ and $F(\ell)$ of $\overline U$ via $F$, we obtain a surface with boundary that we write $\ov{\cal U}$.
Its ``interior'' $\cal U = \ov{\cal U}\setminus \partial\ov{\cal U}$ is a Riemann surface for the complex structure inherited from a neighborhood of $\ell\cup U\cup F(\ell)$; this includes $\ell$ (the gluing is analytic).
The Riemann surface $\cal U$ is biholomorphic to the punctured disk $\D^*$ or equivalently to the half-infinite cylinder $\H/\Z$: this follows from the existence of a quasiconformal homeomorphism that we build below (this construction was not invented by us, it can be found in \cite{Shi} for instance). For an introduction to quasiconformal maps we recommend the following reference: \cite{BF}.

Write $W=X+iY$. Define a homeomorphism
\[G:[0,1]\times(0,+\infty)\to \ell\cup U\cup F(\ell)\] 
as follows:
on each horizontal $G$ is a linear interpolation between $iY$ and $F(iY)$:
\[G(X+iY) = (1-X)iY + XF(iY).\]
Because of the hypothesis \cref{eq:one}, we get that $G$ extends to a neighborhood of its domain to a quasiconformal map that commutes with $F$, see \cite{Shi} for details.\footnote{They use the constant $1/4$ instead of $1/10$.}

\begin{lemma}\label{lem:gdesc}
The map $G$ descends to a quasiconformal homeomorphism $\cal G$ from $\H/\Z$ to $\cal U$:
\[
\begin{tikzcd}
B \arrow[d] \arrow[r, "G"] & U' \arrow[d] \\
\H/\Z \arrow[r, "\cal G"'] & \cal U
\end{tikzcd}
\]
commutes
where $B=[0,1]\times (0,+\infty)$, the vertical arrows are passing modulo $\Z$ and modulo $F$ and $U' = \ell\cup U\cup F(\ell) = \ov U \setminus [0,F(0)]$ (so that $\cal U = U'/F$).
\end{lemma}
\begin{proof}
There is a unique map $\cal G$ satisfying the diagram: the only place in $\H/\Z$ where the projection to $\H/\Z$ has not a unique preimage is the imaginary axis. There, an element has two antecedents: $iY$ and $iY+1$ for some $Y>0$. But then $\cal G$ is uniquely defined there because $G(iY+1) = F(G(iY))$.

In the domain, on can use $(0,1)\times(0,+\infty)$ as a fist chart for (a subset of) $\H/\Z$ and $U$ as a chart in the range. In this chart, $\cal G$ is a $C^1$ diffeomorphism with Beltrami derivative of norm at most $a:=1/9$ so it is $K$-qc with $K = (1+a)/(1-a) = 5/4$.
A neighborhood of $\H\cap i\R$ can be used as a second chart in the domain, and we use a neighborhood of $\ell$ in the range. There, $\cal G$ is given by two $C^1$ diffeomorphisms patched along $i\R$: $G$ on the right of $i\R$ and $T\circ F^{-1}\circ G$ on the left. These two diffeos extend slightly across $i\R$ and coincide there, and are both $5/4$-quasiconformal.
By classical quasiconformal gluing lemmas,\footnote{See for instance Rickman's lemma, Lemma~1.20 in \cite{BF} with $\Phi = $ the patched map, $\phi = G$ and $C=i\R$ or the closed right half plane intersected with an open neighborhood of $i\R$. In this particular case where we glue along a straight line, there are simpler proofs.} $\cal G$ is $5/4$-quasiconformal in this second chart too.
\end{proof}

Since $\cal U$ is quasiconformally equivalent to $\H/\Z$ it is also conformally equivalent to $\H/\Z$. The composition of this isomorphism with the natural projection $\ell \cup U\cup F(\ell) \to\cal U'$ has a lift by the natural projection $\H\to\H/\Z$. Call it
$L: \ell\cup U\cup F(\ell)\to \H$. It has a \emph{holomorphic} extension to a neighborhood $V$ of $\ell\cup U\cup F(\ell)$ such that
$$\qquad L(F(W))=L(W)+1$$
holds in a neighborhood of $\ell$.
We can assume that $V$ is simply connected by taking a restriction if necessary, but we do not assume that it contains $[0,F(0)]$.
However, by Caratheodory's theorem the map $L$ indeed has a \emph{continuous} extension to $[0,F(0)]$ that we call $\bar L$.
By adding a real constant to $L$, we can furthermore assume that 
\[\bar L(0)=0.\]

\begin{figure}
\begin{tikzpicture}
\node at (0,0) {\includegraphics{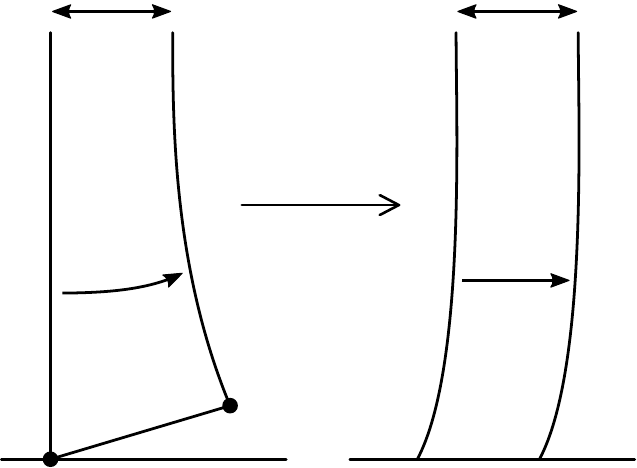}};
\node at (0,.7) {$L$};
\node at (-2.2,2.6) {$\approx 1$};
\node at (2,2.6) {$1$};
\node at (-2.8,-2.6) {$0$};
\node at (-.4,-1.6) {$F(0)$};
\node at (-2.05,1) {$U$};
\node at (-2,-0.95) {$F$};
\node at (-3,0) {$\ell$};
\node at (2,-.1) {$T$};
\end{tikzpicture}
\caption{Gluing.}
\label{fig:L}
\end{figure}

\begin{lemma}\label{lem:gluev3}
Assume that \cref{eq:one} holds.
Then $\forall\, W,W'\in \ell\cup U$
\[ \frac{|\Im (W-W')|-C_1}{A}\leq |\Im \left( L(W)- L(W')\right)| \leq A |\im (W-W')|+C_1\]
for two universal constants $A>1$, $C_1>0$.\\
If $|\Im(L(W)-L(W'))|>C_1$ then $\Im(L(W)-L(W'))$ and $\Im(W-W')$ have the same sign.
\end{lemma}

\begin{lemma}\label{lem:gluev4}
Let $\delta\leq 1/10$ and assume $F$ is a function as above such that
\begin{equation}\label{eq:glueAss}
(\forall W\in\ell) \quad |F(W) - W - 1| \leq \delta \quad\text{and}\quad |F'(W)-1| \leq \delta.
\end{equation}
Then for all $M>0$,
\[\sup_{|\im Z| < M} |L(W)-W| \leq B(M)\delta \]
and
\[\sup_{|\im Z| < M} |L^{-1}(W)-W| \leq B(M)\delta \]
for some continuous non-decreasing function $B(M)>0$ that is universal (i.e.\ independent of $F$ and $\delta$).
\end{lemma}

\proof[Proof of \Cref{lem:gluev3,lem:gluev4}]
The usual approach in this situation (see for instance \cite{Shi} in another context) is to decompose $L$ as $L = H \circ G^{-1}$ where $G$ was defined above.
Write $W=X+iY$.
%On each horizontal $G$ is a linear interpolation between $iY$ and $F(iY)$:
%\[G(X+iY) = (1-X)iY + XF(iY).\]
Let $r(W) = F(W)-W-1$ so that the hypotheses are $|r(iY)|,|r'(iY)|\leq \delta\leq 1/10$.
Then $G(W) = W + Xr(iY)$ whence
\begin{equation}\label{eq:GWW}
|G(W)-W|\leq \delta \leq 1/10.
\end{equation}
A computation of the Beltrami differential $BG$ of $G$ at $X+iY$ gives
\[ \frac{r(iY) - Xr'(iY)}{2+r(iY)+Xr'(iY)} d\ov{z}/dz\]
whose absolute value is  $\leq 2\delta / (2-2\delta) \leq 10\delta/9 \leq 1/9$.

Recall that $L$ induces a conformal isomorphism $\cal L$ between $\cal U$ and $\H/\Z$.
The map $\cal H = \cal L\circ \cal G$ is thus quasiconformal from $\H/\Z$ to itself and has a lift $H := L \circ G$.
Its Beltrami differential $\mu=BH$ coincides with $BG$.
%Its Beltrami differential $\mu=BH$ (an ellipse field) is equal to the push-forward of the null form (circles) by $G$. 
Hence the essential supremum of $BH$ is the same as that of $BG$, hence $\leq 10\delta/9\leq 1/9$.

A quasiconformal map such as $H$ possesses a reflection extensions $\tilde H$ across $\R$ that is quasiconformal and commutes with $z\mapsto\bar z$. The Beltrami differential of $\tilde H$ is a extension of $\mu$ by reflection.

For \Cref{lem:gluev3}, note that the set of quasiconformal maps from $\C/\Z$ to itself with a given bound on the dilatation ratio of its differential, forms a compact family (modulo automorphisms of $\C/\Z$).
In particular, a cylinder of height one has an image of height that is bounded over the family.
It follows that $\tilde H$, satisfies
\[ (\forall W,W'\in \C) \quad |\Im(\tilde H(W)-\tilde H(W'))| \leq A|\Im(W-W')|+C\]
for some universal $A>1$, $C>0$, and the same estimate holds with $\tilde H$ replaced by $\tilde H^{-1}$ because it is also quasiconformal with the same supremum of Beltrami differential.
Hence
\begin{equation}\label{eq:estHt}
 (|\Im(W-W')|-C)/A\leq |\Im(\tilde H(W)-\tilde H(W'))| \leq A|\Im(W-W')|+C.
\end{equation}
The map $H$ being a restriction of $\tilde H$, it satisfies the same inequalities.
Using this and the bound $|G(W)-W|\leq 1/10$, the first claim of the lemma follows with $C_1=C+2A/10$.

Moreover, given $W$, the image of the horizontal closed line in $\C/\Z$ through $W$ is a closed curve winding around $\C/\Z$ and of total height at most $C$.
It follows that if $W'$ is another point such that 
$\Im W'>\Im W$, then $\Im H(W')>\Im H (W)-C$ and if $\Im W'<\Im W$
then $\Im H (W')<\Im H(W)+C$.
So if $|\Im(H(W)-H(W'))|>C$ then $\Im(H(W)-H(W'))$ and $\Im(W'-W)$ have the same sign.
A similar argument proves that if $|\Im(G(W)-G(W'))|>2/10$ then $\Im(G(W)-G(W'))$ and $\Im(W'-W)$ have the same sign.
Now if $|\Im(L(W)-L(W'))|>C+2A/10$ then by the right hand side of \cref{eq:estHt} we get $|\Im(G^{-1}(W)-G^{-1}(W'))|>2/10$ and by the discussion above $\Im(L(W)-L(W'))$ has the same sign as $\Im(G(W)-G(W'))$ which has the same sign as $\Im(W'-W)$.
We proved the last claim of \Cref{lem:gluev3}.

To prove \Cref{lem:gluev4}, note that the estimate on $G$ is global, so we have $|G(W)-W|\leq \delta$ for all $W\in\dom G$ and  $|G^{-1}(W) -W|\leq \delta$ for all $W$ in $\dom G^{-1}$.
We have $L = H\circ G^{-1}$ (first case) and $L^{-1} = G\circ H^{-1}$ (second case) so we now look for an estimate on $H$ that is valid on the cylinder $0<\Im W<M$ for $H^{-1}$ in the second case and on the cylinder $0<\Im W< M+1/10$ for $H$ in the first case.

We can for instance proceed as follows. Normalize $H$ by adding a real constant so that it fixes $0$: i.e. we consider the map $H-H(0)$.
It can be embeded in a holomorphic motion on $\C/\Z$ by straightening $t\times BG$, with $t$ a complex number of modulus small enough so that the essential supremum of $t\times BG$ is $<1$.
By the study above, $|t|<\frac{9}{10\delta}$ is enough.
We normalize the motion by requiring that $0$ stays fixed. 
Consider the hyperbolic distance $d_1$ on the Riemann surface $\C/\Z\setminus\{0\}$ and $d_0$ on the Riemann surface $\D$. The holomorphic motion implies $d_1(W,H(W))\leq d_0(0,\frac{10}{9}\delta t)$. The result follows.
\hfill$\square$\par\medskip

We are now going to give two statements under the supplementary assumption that $F$ extends to $\H$ into a holomorphic funtion satisfying \cref{eq:one} for all $W\in \H$:
\begin{equation}\label{eq:oneH}
(\forall W\in\H) \quad |F(W) - W - 1| \leq 1/10 \quad\text{and}\quad |F'(W)-1| \leq 1/10.
\end{equation}
Since $\H$ is convex, the condition on $F'$ implies that $F$ is \emph{injective} on $\H$.

\Cref{eq:oneH} has a visual consequence: a portion of orbit $F(W),\ldots,F^k(W)$ belongs to a cone of apex $W$, with a central axis which is horizontal, opening to the right and with half opening angle $\theta=\arcsin 1/10$.

\begin{lemma}\label{lem:fdf2}
Assume $F$ satisfies \cref{eq:oneH}.
Then an $F$-orbit can pass at most once in $\ell\cup U$.
\end{lemma}
\begin{proof}
Consider $\gamma_0=i\R$ and
let us extend $F(\ell)$ by a vertical half line going down and stemming from $F(0)$, into  a curve that we call $\gamma_1$.
Since $|F'-1|\leq 1/10$, it follows that $\gamma_1$ can be parameterized by the imaginary part: 
$\gamma_1 = \setof{g(Y)+iY}{Y\in\R}$ for some continuous function $g:\R\to\R$ with at most one non-smooth point, corresponding to the corner $F(0)$.
This function is constant below this point. Above, it satisfies $|g'(Y)|\leq \tan\theta = 1/\sqrt{99}$.
The set $V:\,0\leq X<g(Y)$ is well-defined and contains $\ell\cup U$. It is disjoint from the set $V_+$ of equation $X\geq g(Y)$. It is enough to check that $F(V_+)\subset V_+$ and $F(V)\subset V_+$.
The first inclusion follows from the cone property mentionned above and the bound on $|g'|$. 
For the second inclusion, first note that if $Z\notin \H$ then $F(Z)$ is not defined, so we now assume that $Z\in V\cap \H$.
Link $Z\in U$ with the unique point $Z'\in\ell$ of same imaginary part by the horizontal segment $[Z',Z]$. Then $[Z',Z]\subset \H$ and the image by $F$ of $[Z',Z]$ will not deviate from the horizontal by more than $\arcsin \frac{1}{10}$ and links a point of $F(\ell)$ with $F(Z)$. We conclude using the bound on $|g'|$.
\end{proof}

In the next lemma we use one of Koebe's distortion theorems, which we copy here from \cite{Po} (Theorem~1.3 page~9, equation~(15)): for a univalent map $f$ from $\D$ to $\C$:
\[|f'(0)| \frac{|z|}{(1+|z|)^2} \leq |f(z)-f(0)|\leq |f'(0)| \frac{|z|}{(1-|z|)^2}.\]
Consequence: for a univalent map $f:B(a,R) \to\C$:
\begin{equation}\label{eq:koe}
\frac{|f(z)-f(a)|}{z-a} \left(1-\frac{|z-a|}{R}\right)^2 \leq |f'(a)| \leq \frac{|f(z)-f(a)|}{z-a} \left(1+\frac{|z-a|}{R}\right)^2
\end{equation}

\begin{lemma}\label{lem:Lpun}
Assume that $F$ extends to $\H$ into a holomorphic funtion satisfying \cref{eq:oneH}, and that moreover $F(W)-W-1\tend 0$ as $\im W\to +\infty$.
Then $L'(W)\tend 1$ as $\Im W\to +\infty$ within $\ell\cup U$.
\end{lemma}
\begin{proof}
The method is from \cite{Y}, pages 28--29, simplified here for our setting.
It is a standard trick in this field to extend $L$ to $\bigcup F^k(\ell\cup U)$, $k\in \Z$ so that the relation $L\circ F (W) = T\circ L (W)$ holds whenever both sides are defined. The extension is well-defined because of \Cref{lem:fdf2} and is holomorphic.
Since the set $\ell\cup U$ contains the set defined by the equations $\Im W> 1/10$ and $0\leq\Re(W)<9/10$, it follows by the cone property that the domain of definition of the extension contains the set of equation
\[\Im W>\frac{2}{10}\text{ and }-(\im W-\frac{2}{10})\tan \theta < \Re W <  \frac{9}{10}+(\im W-\frac{2}{10})\tan\theta\]
where $\theta = \arcsin 1/10$. The margin $2/10$ is here to ensure that when $F^k(W)\in U$, we have $\Im F^k(W) > 1/10$ and hence $F^k(W)$ lies above the lower segment $[0,F(0)]$ of $\partial U$.
The extension is injective: indeed if $L(W_1)=L(W_2)$ then consider $k_1,k_2\in\Z$ such that $W'_1:=F^{k_1}(W_1) \in \ell \cup U$ and $W'_2:=F^{k_2}(W_2)\in\ell\cup U$. Then $L(W_1)=L(W'_1)-k_1$ and $L(W_2)=L(W'_2)-k_2$. The points $L(W'_1)$ and $L(W'_2)$ both belong to $L(\ell\cup U)$ which is a fundamental domain for the action of $T$ on $\H$ hence $k_1=k_2$ and thus $L(W_1')=L(W_2')$. So $W'_1 = W'_2$ i.e.\ $F^{k_1}(W_1) = F^{k_2}(W_2)$ and using $k_1=k_2$ again and the injectivity of $F$ we get $W_1=W_2$.
Now for $W\in \ell\cup U$ with $\im W$ big, the extension $L$ is defined on a big disk centered on $W$, injective and satisfies $L(W')=L(W)+1$ where $W' := F(W)$ is close to $W+1$, by hypothesis.
The conclusion then follows using \cref{eq:koe} with $a=W$ and $z=W'$.
\end{proof}

Weaker assumptions are enough and stronger conclusions hold. We only proved here statements that are sufficient to get the main lemma.

\subsection{Iterations and rescalings}\label{sub:iter}

Let $T_\alpha(Z)=Z+\alpha$ and $T=T_1$.
For a holomorphic map $F$ commuting with $T$, defined on a domain containing an upper half plane and satisfying $F(Z) = Z+\alpha +o(1)$ as $\Im Z\to+\infty$ we call $\alpha$ its \emph{rotation number} and let it be denoted by $\alpha(F)$.
The following properties are elementary and stated without proof.

Let $\alpha\in\R$ and assume that
\[\|F-T_{\alpha(F)}\|_\infty<K|\alpha(F)-\alpha|.\]
\begin{itemize}
\item Then for $k>0$: $\alpha(F^k)=k\alpha(F)$ and 
\[\|F^k-T_{\alpha(F^k)}\|_\infty< K|\alpha(F^k)-k\alpha|.\]
\item
Let $a\in\C$, $b\in\R$ with $b>0$, write $\lambda(Z)=bZ+a$ and
$G=\lambda\circ F\circ \lambda^{-1}$. Then $\alpha(G) = b\alpha(F)$ and
\[\|G-T_{\alpha(G)}\|_\infty< K|\alpha(G)-b\alpha|.\]
\end{itemize}

Assume instead that
\[\|F'-1\|_\infty < \exp(K|\alpha(F)-\alpha|)-1.\]
\begin{itemize}
\item Then
\[\|(F^k)'-1\|_\infty< \exp(K|\alpha(F^k)-k\alpha|)-1,\]
\item and
\[\|G'-1\|_\infty = \|F'-1\|_\infty.\]
\end{itemize}

\subsection{Reminder on continued fractions}\label{sub:remind}

We state a few classical properties for reference.
Let $a_0\in\Z$ and $a_n\in \N^*$ for $n>0$.
The notation $[a_0;a_1,\ldots,a_n] = a_0 + 1/(a_1+1/(\ldots+1/a_n))$ is often used with integers only but we will use it too with the last entry being a real number: $[a_0;a_1,\ldots,a_n,x] = a_0 + 1/(a_1+1/(\ldots+1/(a_n+1/x)))$.
If we write $p_n/q_n=[a_0;a_1,\ldots,a_n]$ in lowest terms (with $p_{-1}/q_{-1} = 1 / 0$) then
\[\alpha := [a_0;a_1,\ldots,a_n,x] = \frac{p_nx+p_{n-1}}{q_nx+q_{n-1}}\]
and conversely
\[x=-\frac{q_{n-1}\alpha-p_{n-1}}{q_{n}\alpha-p_n}.\]
Also, $p_{n-1}q_n-p_nq_{n-1} = (-1)^n$,
\[q_{n}\alpha-p_n = \frac{(-1)^n }{q_n x + q_{n-1}}.\]
and
\[q_{n-1}\alpha-p_{n-1} = \frac{(-1)^{n-1} }{q_n + q_{n-1} x^{-1}}\]
from which we can get the following classical inequality (shifting the index $n$):
\[|q_{n}\alpha-p_{n}|\leq \frac{1}{q_{n+1}}.\]
If 
\[\beta=[a_0;a_1,\ldots,a_n,y]\]
then 
\[\beta-\alpha = (-1)^{n+1}\frac{y-x}{(q_n x+q_{n-1})(q_n y+q_{n-1})}.\]
Also there are the famous induction relations, for $n\geq 1$:
\[p_{n}=a_n p_{n-1}+p_{n-2},\]
\[q_{n}=a_n q_{n-1}+q_{n-2}.\]

In this article we call continued fraction expansion the notation $[a_0;a_1,\ldots]$ where the sequence $a_n$ is finite or infinite and where $a_0\in\Z$ and $a_n\geq 1$ for $n\geq 1$.
If $\alpha\in\R\setminus\Q$ then it has only one continued fraction expansion and it is infinite.
If $\alpha\in\Q$ then it has exactly two continued fraction expansions: $[a_0;a_1,\ldots,a_s,1]$ and $[a_0;a_1,\ldots,a_s+1]$.

\subsection{Renormalization}\label{sub:direct}

\subsubsection{Foreword}
The renormalization procedure we describe here is a variant of what is usually done.

Consider a map $F\in \cal S(\alpha)$.
Usually a fundamental domain $U$ is defined, bounded by $\ell\cup[iy_0,F(iy_0)]\cup F(\ell)$ where $\ell$ is the vertical half line from $iy_0$ to $+i\infty$ and $y_0>0$ is a real chosen big enough to ensure good behaviour of the construction. Then a ``return'' map from $T^{-1}(\ell\cup U)$ to $\ell\cup U$ is defined. Conjugacy through the gluing basically defines the renormalization.

Usually this procedure is iterated a great number of times to give information on high iterates of the original map. 
Here, proximity to a rotation allows to bypass this and apply a one-step renormalization procedure directly to the high iterates.

\subsubsection{Construction}

Consider $F\in \cal S(\alpha)$. Let $p_{k-1}/q_{k-1}$, $p_k/q_k$ and $p_{k+1}/q_{k+1}$ be three successive convergents of $\alpha$ with $k\geq 0$.
Implicitly $\alpha\neq p_k/q_k$.
The construction described below depends on the choice of $k$ and will also depend on the choice of a positive constant $y_0$. In this article we call $k+1$ the \emph{order} of the renormalization.

Write 
\begin{eqnarray*}
J &=& T^{-p_{k-1}}\circ F^{q_{k-1}},\\
H &=& T^{-p_k}\circ F^{q_k} .
\end{eqnarray*}
The symbol $J$ refers to a \emph{jump} and $H$ to a \emph{hop}: indeed when $\Im Z$ is big enough, $J$ moves points by a bigger\footnote{There is one exception: $\alpha = m = [m-1; 1]$ for some $m\in\Z$ and we choose $k=0$. We then get $\beta'=-1$ and $\beta = 1$. This case is not necessary for our main result but all we state here holds for it too, except the claim that $|\beta|<|\beta'|$.} amount than $H$.
Let
\begin{align*}
\beta' &=\alpha(J) = q_{k-1}\alpha-p_{k-1},\\
\beta &= \alpha(H) = q_{k}\alpha-p_{k}.
\end{align*}
By the theory of continued fractions, $1/2\leq q_{k+1}|\beta|\leq 1$, hence $\beta\neq 0$ and $\beta$ and $\beta'$ both tend to $0$ as $k\to+\infty$.
Moreover the sign of $\beta$ alternates: it coincides with the sign of $(-1)^{k}$ and $\beta'$ has the opposite sign.

\begin{remark} In the particular case $k=0$, we have $p_0=a_0$, $q_0=1$ and by convention $p_{-1}=1$, $q_{-1}=0$ so $J=T^{-1}$ and $H = T^{-a_0}\circ F$. Then the construction is essentially the classical renormalization of \cite{Y}, and we call it order $1$ renormalization according to the convention above.
\end{remark}

Assume that we have identified a height $y_0$ such that the following statements hold:
\begin{itemize}
\item the domain of definition of $H$ contains $\Im Z>y_0$; as a consequence the domain of $J$ also does;
\item $\forall Z\text{ with }\Im Z>y_0$:
\begin{align*}
\qquad |H(Z)-Z-\beta|&\leq |\beta|/10\text{,}
& |H'(Z)-1|&\leq 1/10,
\\
\qquad |J(Z)-Z-\beta'|&\leq |\beta|/10
\text{,}\footnotemark
& |J'(Z)-1|&\leq 1/10
\end{align*}%
\end{itemize}
\footnotetext{This is not a typographic mistake: we want $\beta'$ on the left hand side of the inequality and $\beta$ on the right hand side.}
In particular if $\beta>0$ we have $\Re H(Z)>\Re Z$ and $\Re J(z) < \Re Z$, and if $\beta<0$ then it is the opposite: $\Re H(Z)< \Re Z$ and $\Re J(z) > \Re Z$.

\begin{remark}
There always exists such a height $y_0$,\footnote{Maps in $\cal S(\alpha)$ are close to $Z\mapsto Z+\alpha$ when $\im Z$ is big, see for instance \cite{Y} page~26.} and part of the work in further sections will be to get some control over it.
\end{remark}

\begin{notation}
Since the construction involves many changes of variables we adopt the following notation: if $Z\mapsto \lambda(Z)$ is a change of variable then instead of denoting the new variable $\lambda(Z)$ or $Z'$ or $W$ we may choose the notation $Z^\lambda$. We speak of the $Z^\lambda$-space, instead of the $W$-space or such. If a map acts on $Z^\lambda$-space we may choose to use the notation $F^\lambda$. Similarly a set in $Z^\lambda$-space may be denoted by $S^\lambda$.
\end{notation}

We now define the following change of variable $\lambda$:
\begin{align}\label{eq:lambda}
\begin{split}
&\text{If }\beta>0\text{ let }\lambda(Z) = (Z-iy_0)/\beta.\\
&\text{If }\beta<0\text{ let }\lambda(Z) =\ov{(Z-iy_0)}/\beta.
\end{split}
\end{align}

By \Cref{sub:iter} the map $H^\lambda = \lambda \circ H\circ \lambda^{-1}$ then satisfies \cref{eq:one} stated in \Cref{sub:glue}, with $H^\lambda$ in place of $F$.
The sets $\ell$ and $U$ constructed there will be denoted here by $\ell^\lambda$ and $U^\lambda$ because they live in $Z^\lambda$-space, so that we can call $\ell$ and $U$ their images by $\lambda^{-1}$, which live in $Z$-space.
Then
\[\ell = i(y_0,+\infty) \text{ and } \partial U = \ell\cup[iy_0,H(iy_0)]\cup H(\ell)\]
If $\beta<0$ then $U$ sits on the left of $\ell$, otherwise it is on the right.
%We now assume that $y_1$ has been chosen to ensure that:
%\begin{itemize}
%\item both $\ell\cup U$ and $J(\ell\cup U)$ sit above height $y_0$.
%\end{itemize}

A portion of $H$-orbit $Z$, $H(Z)$, \ldots, $H^n(Z)$ that stays above $y_0$ (except maybe at the last iteration) can hit $\ell\cup U$ at most once: this follows from \Cref{lem:fdf2} applied to the restriction to $\H$ of $H^{\lambda}$. 

We now define a return map $R$, defined on a subset of $\ell\cup U$ and taking values on $\ell\cup U$. For $Z\in\ell\cup U$:
\begin{itemize}
\item If there is some $n\geq 0$ such that $Z$ and
$J(Z)$, $H(J(Z))$, $H^2(J(Z))$, \ldots,  $H^{n-1}(J(Z))$ are all above $y_0$ and $H^n(J(Z))\in \ell\cup U$ then we let $R(Z)=H^n(J(Z))$. Below we temporarily write $n(Z)$ this unique value of $n$.
\item Otherwise we let $R$ be undefined at $Z$.
\end{itemize}

Recall that $\ov{\cal U}$ refers to the quotient $\ov U / H$. It is a surface with boundary and its ``interior'' is denoted by $\cal U$ and is a Riemann surface. We have $\cal U = (\ell\cup U\cup H(\ell))/H$ and the canonical projection is a bijection from $\ell \cup U$ to $\cal U$.

\begin{lemma}
The map $R$ is injective. Passing to the quotient $\cal U$ (i.e.\ conjugating by the canonical projection $\ell\cup U \to \cal U$), $R$ becomes continuous and better: is the restriction of a holomorphic map to $\dom R$.\footnote{By definition, holomorphic maps are defined on open sets. The domain of $R$ may fail to be open near points of $\ell$ in the quotient for subtle reasons in the definition of $R$.}
\end{lemma}
\begin{proof}
Injectivity: Assume that $R(Z_1)=R(Z_2)$ for $Z_1,Z_2\in\ell\cup U$ with
$R(Z_1)=H^{n_1}(J(Z_1))$ and $R(Z_2)=H^{n_2}(J(Z_2))$.
Since $F$ is injective, it follows that $J$ and $H$ are injective too. Hence $J(Z_1)$ and $J(Z_2)$ belong to the same $H$-orbit. Up to permuting them, we can assume $J(Z_1)=H^k(J(Z_2))$ for some $k\geq 0$. Now
$H^k(J(Z_2)) =J(H^k(Z_2))$, whence $Z_1=H^k(Z_2)$ by injectivity of $J$. Using \Cref{lem:fdf2} we get $k=0$.

Continuity and holomorphy:\footnote{Let us give a heuristic justification. The map $J$ has an essentially well-defined action on the orbits of the restriction of $H$ to $\Im Z>y_0$ because these two maps commute; the quotient of the gluing can be seen as a subset of the space of orbits and its analytic structure is such that $Z\mapsto\text{orbit}(Z)$ is holomorphic.
Now there are some problems in this approach since the space of orbits is not that well defined, or does not have such a nice topological structure, and the action of $J$ is not so well defined.}
Let us use two charts for the analytic structure on $\cal U$. The first chart is the union of $U$ and of a small enough connected open neighborhood $V$ of $\ell$. Recall that $\ell$ does not contain its endpoint, so we can take a neighborhood of size that shrinks to $0$ near $y_0$. The second chart is $U\cup H(V)$. They can be glued along $V$ using $H$ to give a complex dimension one manifold $W/H$ where $W=V\cup U\cup H(V)$, and this manifold is canonically isomorphic to $\cal U$.
We choose $V$ small enough so that $V$ and $H(V)$ are disjoint, so that $V\subset U\cup \ell \cup H^{-1}(U)$ and so that orbits of the restriction of $H$ to $\Im Z>y_0$ intersect $W$ in exactly one point or in exactly two in consecutive iterates, one in V, the other in $H(V)$, and finally so that the image of $W\cap \dom J$ by $J$ does not intersect $H(V)$. Then 
for any representative $Z$ of a point of $(\dom R)/H$, for any $n\geq 0$ such that $H^n(J(Z))\in W$, then $H^n(J(Z))$ is a representative of $R(Z)$. The result follows.
\end{proof}

Recall that in \Cref{sub:glue} we associated a map $L$ from $\ell^\lambda\cup U^\lambda$ to $\H$ via gluing, uniformization, then unfolding. The composition $L\lambda$ goes from $\ell\cup U$ to $\H$. We omit the symbol ``$\circ$'' in $L\circ\lambda$ for more compact expressions.

\begin{lemma}\label{lem:hp}
The domain $\dom R$ contains every point in $\ell \cup U$ of high enough imaginary part.
The set $\Z+L\lambda(\dom R)$ contains some upper half plane.
\end{lemma}
\begin{proof} We assume $\beta>0$, the other case being similar.
The first time a portion of $H$-orbit passes from the left (strictly) to the right of the imaginary axis (inclusive), then a sufficient condition for the point to belong to $\ell\cup U$, is that its imaginary part be $>\max(\Im (iy_0), \Im H(iy_0))$.
Now for $Z\in \ell\cup U$ with high enough imaginary part, $J(Z)$ is defined, lies on the left of $i\R$ and $|\Re(J(Z))|$ is bounded, for instance by  $|\beta'|+\beta/10$.
Applying $H$ to a point above height $y_0$ increases the real part by at least $9\beta/10$ while the imaginary part changes by at most $\beta/10$.
From there the details are left to the reader.

For the second claim consider a point $W^{L}\in \H$ and let us translate it by an integer so that $W^{L}\in L(\ell^\lambda\cup U^\lambda)$, which is possible since $L(\ell^\lambda\cup U^\lambda)$ is a fundamental domain for $\H/\Z$.
Let us apply the right hand side of \Cref{lem:gluev3}, and the sign comparison claim of the same lemma, to $W:=L^{-1}(W^{L})$ and $W':=\eps$ and let $\eps\to 0$. Noting that $\Im L(W')\tend 0$ we get
$\Im W^L>C_1$ $\implies$ $\Im W \geq (\Im W^L-C_1)/A$ $\implies$ $\Im (L\lambda)^{-1}(W^L) \geq y_0+(\Im W^L-C_1)\beta/A$. The second claim follows from this and the first claim.
\end{proof}

Let
\[H_0 = \inf \setof{H>0}{\Im Z>H\implies Z\in\Z+L\lambda(\dom R)}\]
and let
\begin{equation}\label{eq:Lambda}
\Lambda(Z) = L\lambda(Z)  - iH_0.
\end{equation}
Given $Z^\Lambda \in \H$ there is a unique $k=k(Z^\Lambda)\in \Z$ such that $Z^\Lambda + k\in \Lambda(\ell\cup U)$. Let $Z = \Lambda^{-1}(Z^\Lambda+k)$.
By definition of $H_0$, we have $Z\in \dom R$, in particular there exists some (unique) $m = m(Z^\Lambda)\geq 0$ such that $H^m\circ J(Z) = R(Z) \in \ell\cup U$. Let
\begin{equation}\label{eq:RF}
\cal RF(Z^\Lambda)=\Lambda(R(Z))-m-k.
\end{equation}
If $\Im(Z^\Lambda)\in(-H_0,0]$ we choose to declare $\cal RF$ undefined at $Z^\Lambda$, even though the procedure above may reach fruition.
The map $\Lambda$ conjugates (a restriction of) $R$ to $\cal R F \bmod \Z$.

\begin{claim*}
The map $\cal RF$, which we defined on $\H$ and takes values in ``$\Im Z^\Lambda>-H_0$'', is continuous and better: holomorphic. 
\end{claim*}
\begin{proof}
Indeed consider a holomorphic extension $\tilde L$ of $L$ to a neighborhood of ${U'}^\lambda := \lambda(U')$ with $U'=\ell\cup U\cup H(\ell)$, which satisfies
$\tilde L(H^\lambda(Z^\lambda))=\tilde  L(Z^\lambda)+1$
for $Z^\lambda$ in a neighborhood of $\ell^\lambda:=\lambda(\ell)$ (see the beginning of \Cref{sub:glue}) and let $\tilde \Lambda = \tilde L\lambda-iH_0$.
Then
\begin{equation}\label{eq:funcLtilde}
\tilde \Lambda(H(Z))=\tilde \Lambda(Z)+1
\end{equation}
holds in a neighborhood of $\ell$.
Consider $k$ and $m$ as in \cref{eq:RF}.
It is enough to check that in a neighbourhood of any point $Z_0^\Lambda\in\H$, the formula \[\cal RF(Z^\Lambda)=\tilde\Lambda\circ H^m\circ J\circ \tilde\Lambda^{-1}(Z^\Lambda+k)-m_0-k_0\] is locally valid, with $m_0 = m(Z_0^\Lambda)$ and $k_0=(Z_0^\Lambda)$.
If nearby $Z^\Lambda$ have a different value of $k$ in \cref{eq:RF}, this means the initial $Z^\Lambda$ belongs to $\ell^\Lambda :=\Lambda(\ell)$, $m_0>0$ and the nearby values of $Z^\Lambda$ have a value of $k$ that equals $k_0$ or $k_0+1$. In the latter case we can use \cref{eq:funcLtilde} and get $\Lambda^{-1}(Z^\Lambda+k(Z^\Lambda)) = H(\tilde\Lambda^{-1}(Z^\Lambda+k_0))$.
In both cases the following holds locally:
$\cal RF(Z^\Lambda) = \Lambda \circ H^{m+k-k_0} \circ J \circ \tilde \Lambda^{-1}(Z^\Lambda+k_0) -m-k$, which we rewrite
\[ \cal RF(Z^\Lambda) = \Lambda \circ H^{m_0+\delta} \circ J \circ \tilde \Lambda^{-1}(Z^\Lambda+k_0) -m_0-k_0 -\delta \]
with $\delta=(m-m_0)+(k-k_0)$.
Similarly if local values of $\delta$ differ, then $H^{m_0}\circ J(\Lambda^{-1}(Z_0^\lambda+k_0)) \in \ell$ and $\delta = 0$ or $1$. Again, we can use \cref{eq:funcLtilde}.
\end{proof}

The map $\cal RF$ commutes with $T$ and satisfies 
\begin{equation}\label{eq:rnRF}
\cal RF(Z^\Lambda) = Z^\Lambda  + \alpha' + o(1) 
\end{equation}
as $\im Z^\Lambda \to +\infty$, 
where
\[\alpha'= \alpha(\cal RF) = \frac{\beta'}{\beta} = \frac{q_{k-1}\alpha-p_{k-1}}{q_k\alpha-p_k} = - [a_{k+1};a_{k+2},a_{k+3},\ldots].\]

\Cref{eq:rnRF} follows from $L'(W)$ tending to $1$ as $\Im W\to+\infty$ by \Cref{lem:Lpun}, while the domain $\ov U$ of $L$ has a bounded projection to the real axis.
Indeed $J$ moves points by essentially $\beta'$, $\lambda$ is a translation followed by a rescaling\footnote{Whatever the sign of $\beta$ is, we have $\Re \lambda(Z) = (\Re Z)/\beta$ and $\Im \lambda (Z) = (\Im Z-y_0)/|\beta| $.} by $1/\beta$ and in the definition of $\cal R$, we compensated the effect of $k$ and $m$.
See also \cite{Y} where a finer estimate on $L$ is given.
As an alternative one can use the invariance of the rotation number of holomorphic maps by homeomorphisms (\cite{N,CGP}).

\begin{lemma}\label{lem:h2}
Assume that $Z\in \ell\cup U$ and that $\Lambda(Z)\in K(\cal RF)$. Then $Z\in K(F)$.
\end{lemma}
\begin{proof}
First note that the hypothesis implies that some $Z+n$, $n\in\Z$, can be iterated infinitely many times by $R$.
Since $F$ and $T$ commute, and by definition of $R$, we have
$R(Z+n) = T^k F^m(Z)$ for some $k\in\Z$ and $m\in\N$ with $m>0$ or $m=0$ in some rare exceptional cases that may happen if we do order $1$ renormalization.
But in this case $R(Z)=Z-1$ so this cannot happen twice in a row, because $T^{-2}(\ell \cup U)$ is disjoint from $\ell \cup U$.
\end{proof}

\subsection{Proof of \Cref{lem:main2} for $\alpha\equiv 0 \bmod \Z$}\label{sub:a0}

As already noted, we can assume $\alpha=0$. The number $\alpha$ has the following two continued fraction expansions:
\[0 = [0] = [-1;1].\]
They respectively give
$\alpha_n = [0;n+1+\sqrt2] = 1/(n+1+\sqrt2)$ or $\alpha_n = [-1;1,n+1+\sqrt2] = \ldots = -1/(n+2+\sqrt 2)$.
If $\alpha_n<0$ we can conjugate the sequence $F_n$ by the reflection of vertical axis: $X+iY\mapsto -X+iY$, and proceed then exactly as below, so we assume now that
\[\alpha_n>0.\]

We will apply order $1$ renormalization, i.e.\ proceed to the construction of \Cref{sub:direct} with $k=0$ and
\[F=F_n\]
Then $p_0=a_0=0$, $q_1=1$, $\beta=\alpha_n$, $H=F_n$, $p_{-1}=1$, $q_{-1}=0$, $\beta'=-1$ and $J = T^{-1}$.

By assumption on  \Cref{lem:main2}:
\[ (\forall Z\in \H)\qquad |F_n(Z)-Z-\alpha_n|\leq K\alpha_n
\]
By the Schwarz-Pick inequality this implies:
\[ |F_n'(Z)-1|\leq \frac{K\alpha_n}{2\im Z}. \]
So there exists $\epsilon_n\tend 0$ such that
\[\sup_{\Im Z>\eps_n}|F_n'(Z)-1|\leq 1/10. \]
As explained in \Cref{subsub:defs} one can also write $F(Z)=Z+\alpha+h(e^{2\pi i Z})$ with $h$ a holomorphic function mapping $0$ to itself. Schwarz's inequality thus implies
\[ |F_n(Z)-Z-\alpha_n|\leq K\alpha_ne^{-2\pi\Im Z}.
\]

We take
\[y_0  = \max(\eps_n,\log(10K)/2\pi)\]
so that
\[\Im Z> y_0 \implies |F_n(Z)-Z-\alpha_n|\leq \alpha_n/10 \quad \text{and}\quad |F_n'(Z)-1| < 1/10.\]
%We then can take the following value of $y_1$ to ensure the conditions of the construction of \Cref{sub:direct}:
%\[y_1 = y_0 + \alpha_n/10.\]
In the notation $y_0$ and many of the notations that follow we omit the index $n$ for better readability.
The construction yields two sets $\ell=i(y_0,+\infty)$ and $U$,
a map $L\lambda:\ov U \to \ov \H$
where
\[\lambda(Z) = \frac{Z - iy_0}{\alpha_n},\]
and a return map $R$ from $\ell\cup U$ to itself. It also introduces:
a constant $H_0$ defined as the smallest $H\geq 0$ such that $\Z+L\lambda(\dom R)$ contains $``\Im Z>H"$; the map
\[\Lambda(Z)=L\lambda(Z)-iH_0;\]
and finally the renormalized map $\cal RF_n$, which is a modification of the restriction to $\H$ of the conjugate of $R$ by $\Lambda$.
By the properties stated in \Cref{sub:direct}, including \cref{eq:rnRF}, we have $\cal RF_n \in S(\alpha')$ with $\alpha' = \pm \sqrt2$. It follows that
\[h(\cal RF_n)\leq C_{\sqrt2},\]
see \Cref{lem:s2}.

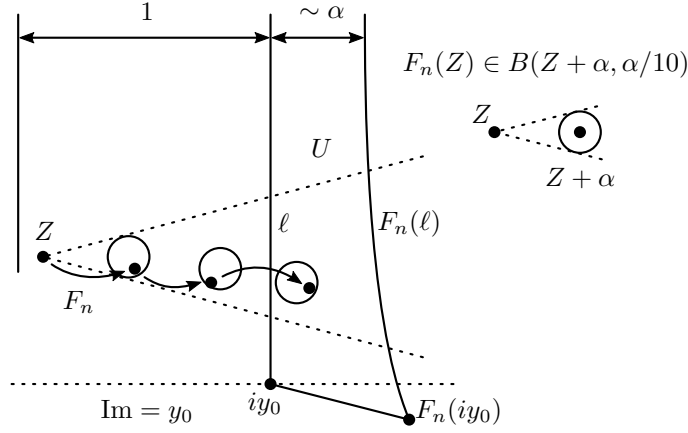
\begin{figure}
\begin{tikzpicture}
\node at (-2.4,0) {\includegraphics{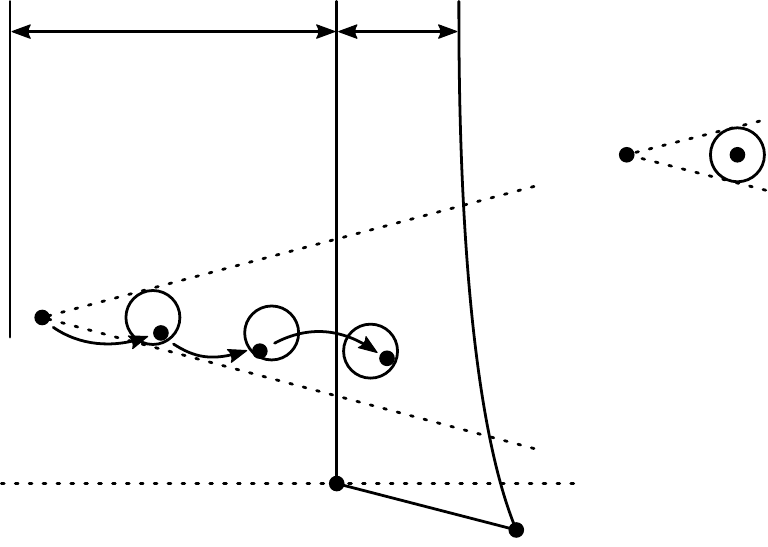}};
\node at (-5.85,-.1) {$Z$};
\node at (-.1,1.45) {$Z$};
\node at (1.2,.6) {$Z+\alpha$};
\node at (.75,2.1) {$F_n(Z)\in B(Z+\alpha,\alpha/10)$};
\node at (-4.5,2.8) {$1$};
\node at (-2.2,2.8) {$\sim \alpha$};
\node at (-2.95,-2.4) {$iy_0$};
\node at (-4.5,-2.5) {$\Im =y_0$};
\node at (-.4,-2.5) {$F_n(iy_0)$};
\node at (-2.2,1) {$U$};
\node at (-5.4,-1.05) {$F_n$};
\node at (-2.7,.0) {$\ell$};
\node at (-1.05,.0) {$F_n(\ell)$};
\end{tikzpicture}
\caption{The orbit of $Z$ remains in a cone, at least as long it stays above height $y_0$. In this picture we exagerated the angle of the cone.}
\label{fig:visit}
\end{figure}

For any $Z$ with $\im Z>y_0$, the point $F_n(Z)$ lies in a horizontal cone of apex $Z$ and with half opening angle
\[\theta = \arcsin(1/10).\]
The orbit will thus stay in that cone as long as the previous iterates all lie above $y_0$, see \Cref{fig:visit}.

Let us give a more explicit version of \Cref{lem:hp}:
\begin{lemma*}
Every point in the strip $\Re Z\in[-1,0[$ and $\Im Z > y_0 + \tan \theta$ has an orbit by $F_n$ that eventually passes the imaginary axis, i.e. $\Re F_n^k(Z)\geq 0$. The first time it does, $\Im F_n^k(Z) \geq \Im Z  - \alpha_n/10 - \tan \theta$. Before, it stays above $y_0$.

\end{lemma*}
\begin{proof} By the cone condition, it follows by induction on $i$ that $F_n^i(Z)$ stays above $y_0$ as long as it belongs to the strip. By assumption when we iterate a point above $y_0$, the real part increases by at least $9/10 \alpha_n$ so we know the orbit will eventually pass the imaginary axis. Just before it was above $\Im Z-\tan \theta$ and at the next iterate the imaginary part decreases at most by $\alpha_n/10$.
\end{proof}

If the first iterate $F_n^k(Z)$ passing the imaginary axis in the lemma above satisfies $\Im F_n^k(Z) > \max(y_0,\Im F_n(y_0))$ then* $F_n^k(Z)\in \ell\cup U$.

(*)\label{star1} For a justification of this claim, consider the horizontal segment from $F_n^{k-1}(Z)$ to $\ell$. Its image is a curve with tangent deviating less that $\theta<\pi/2$ from the horizontal, whereas $F(\ell)$ has a tangent that deviates less than $\theta$ from the vertical, so $F(\ell)$ is contained in $\setof{z\in\C}{|\arg(z-F^k_n(Z))|<\pi/2+\theta}$. It follows that $F_n^k(Z)$, can be linked to $\ell$ by a horizontal segment going to the left and that does not cross the other boundary lines of $U$.

Now $\Im F_n(y_0)\leq y_0+\alpha_n/10$.
By the lemma above, $\dom R$ contains every point in $\ell\cup U$ of imaginary part strictly larger than $y_1$ with
\[y_1=y_0+2\alpha_n/10 +\tan\theta.\]

We will apply \Cref{lem:gluev3} to $L$. 
It introduced constants $A>1$ and $C_1>0$.
\begin{lemma*}
For $Z\in \ell\cup U$ and $\Im Z>y_2$ with
\[y_2= y_1 + \alpha_n(\frac1{10}+C_1+A\max(C_1,C_{\sqrt 2}))\]
then $Z\in K(F_n)$.
\end{lemma*}
\begin{proof}
Consider such a $Z$.
We apply \Cref{lem:gluev3} to specific values of $W$ and $W'$: consider a point $V_0$ in the boundary of $\Z+L\lambda(\dom R)$ and maximizing the imaginary part, so $\Im V_0 = H_0$.
By adding a (possibly negative) integer, we may assume that $V_0\in L\lambda(\dom R)$.
Let $Z_0=(L\lambda)^{-1}(V_0)\in \ell\cup U$. If $Z_0\in U$ then $\Im Z_0\leq y_1$ otherwise it would be in the interior of $\dom R$. If $Z_0\in \ell$ then $\Im Z_0\leq y_1+\alpha_n/10$ otherwise $L\lambda(Z_0)$ belongs to the interior of $\Z+L\lambda(\dom R)$.
In all cases, $\Im Z_0 \leq y_1+\alpha_n/10$.
We take $W=\lambda(Z)$ and $W' = \lambda^{-1}(V_0)$.
Note that $\Im W - \Im W' > (y_2 - y_1 -\alpha_n/10)/\alpha_n >0$.
By the left hand inequality in the \Cref{lem:gluev3}:
$|\Im L(W)-\Im L(W')|\geq (|\Im W-\Im W'|-C_1)/A \geq \max(C_1 ,C_{\sqrt 2}) \geq C_1$.
By the second claim in \Cref{lem:gluev3} we get that $\Im  L(W) > \Im L(W')$. Now
$\Im L(W)>\Im L(W') + C_{\sqrt 2}$ i.e.\ $\Im \Lambda(Z)>C_{\sqrt 2}$
hence $\Lambda(Z)\in K(\cal RF_n)$, hence $Z\in K(F_n)$ by \Cref{lem:h2}.
\end{proof}

Now this construction could have been carried out on the conjugate of $F_n$ by any horizontal translation $Z\mapsto Z+x$, which amounts to replace the origin $iy_0$ of the line $\ell$ by $x+iy_0$. In particular every point with imaginary part $\geq y_0$ is contained in the set $\ell\cup U$ associated to an appropriate choice of $x$. Hence
\[h(F_n)\leq y_2.\]

Putting everything together, we have proved that
$h(F_n) \leq \alpha_n(1/10+C_1+A\max(C_1,C_{\sqrt 2})) +2\alpha_n/10 +\tan\theta + \max(\eps_n,\log(10K)/2\pi)$ where $\theta=\arcsin 1/10$.
Since $\alpha_n\tend 0$ this gives:\footnote{In fact, $\alpha_n\leq 1$ is enough.}
\[\limsup_{n\to+\infty} h(F_n)\leq C(K) := C_0 + \frac{\log K}{2\pi}\]
for some universal constant $C_0>0$.

\subsection{Improvement through renormalization for maps tending to a non-zero rotation}\label{sub:improve}

Consider $\alpha\in\R$ with $\alpha\notin \Z$ and $\alpha_n\in\R$ with $\alpha_n\tend\alpha$.
Consider a sequence $F_n\in \cal S(\alpha_n)$ and assume $F_n\tend T_\alpha$ uniformly on $\H$ when $n\to+\infty$ where
\[T_\alpha(Z)= Z+\alpha.\]
The first statement we give does not need a Lipschitz type assumption on how fast this convergence occurs.

Consider a renormalization as per \Cref{sub:direct}: it involves the choice of $k$ such that $\alpha$ has a continued fraction\footnote{$\alpha$ has one or two c.f.\ expansions, see \Cref{sub:remind}.} of which $[a_0;a_1,\ldots,a_{k+1}]$ is an inital segment. We let $k$ be constant, i.e.\ independent of $n$.

To proceed with the construction of the renormalization, we need to choose $y_0$ such that the conditions of \Cref{sub:direct} are satisfied.
We will use the notation $H$ and $J$ as in that section, i.e.\ without the index $n$:
\begin{eqnarray*}
J &=& T^{-p_{k-1}}\circ F_n^{q_{k-1}},\\
H &=& T^{-p_k}\circ F_n^{q_k} .
\end{eqnarray*}
Let
\begin{align*} & \beta_n = \alpha(H) = q_k\alpha_n-p_k,
\\ & \beta'_n = \alpha(J) =q_{k-1}\alpha_n-p_{k-1}
\end{align*}
be their respective rotation numbers and let
\[\beta = q_k\alpha-p_k\]
\[\beta'=q_{k-1}\alpha-p_{k-1}\]
be their respective limits.
We have $0<|\beta|<|\beta'|\leq 1$.
\[\beta_n \underset{n\to\infty}{\tend} \beta \neq 0 \text{ and } \beta'_n \underset{n\to\infty}{\tend} \beta' \neq 0 \]

Note that, as $n\to+\infty$, $H\tend T_{\beta}$ and $J\tend T_{\beta'}$ hence
\[ \|H-T_{\beta_n}\|_\infty\ntoi\tend 0\text{ and }\|J-T_{\beta'_n}\|_\infty\ntoi\tend 0
.\]
Since $\beta_n\tend\beta\neq 0$ it follows that for $n$ big enough we have $\forall Z\in \dom H$, $|H(Z)-Z-\beta_n| < |\beta_n|/10$ and $\forall Z\in \dom J$,  $|J(Z)-Z-\beta'_n| < |\beta_n|/10$.
Also, for all $\eps>0$, $H'\tend 1$ and $J'\tend 1$ as $n\to+\infty$ both uniformly on the subset of $\H$ defined by the equation $\Re Z>\eps$. Whence the existence of $\eps_n\tend 0$ and $\delta_n\tend 0$ with $\delta_n<1/10$ such that for $n$ big enough, $\Im Z>\eps_n$ $\implies$ $Z\in\dom H$ (hence $Z\in\dom J$) and
\[|H'(Z)-1|<\delta_n,\]
\[|J'(Z)-1|<\delta_n.\]

Thus we can take $y_0=\eps_n$, assuming $n$ big enough. Note that $y_0\tend 0$ as $n\to+\infty$. 
%Similarly, we can take $y_1=y_0+\|H-T_{\beta_n}\|_\infty +\|J-T_{\beta'_n}\|_\infty \tend 0$. 
The exact value of $y_0$ is not so important, what matters is that it tends to zero:
\[y_0\underset{n\to\infty}{\tend} 0\]
%\text{ and }y_1\underset{n\to\infty}{\tend} 0.\]

Then \Cref{sub:direct} associates an order $k+1$ renormalization $\cal RF_n$ to the pair of maps $H$, $J$, via a return map $R$ and a straightening $L\lambda \bmod \Z$ of a Riemann surface $\cal U = (\ell\cup U\cup H(\ell))/H$, where $\lambda$ 
is a change of variable that takes the form $\lambda(Z)=(Z-iy_0)/\beta_n$ or $\ov{(Z-iy_0)}/\beta_n$ (it depends on $k$).

\begin{lemma}\label{lem:improve}
We have
\[\limsup h(F_n)\leq |\beta|\limsup h(\cal RF_n).\]
The same statement holds with $\limsup$ replaced by $\liminf$.
\end{lemma}
\begin{proof}
It is enough for both statements to prove that if $h_0\geq 0$ and if we have a subsequence $n\in I\subset\N$ and points $Z_n\in\H$ with $Z_n\notin K(F_n)$ but $\Im Z_n \geq h_0$ then $\liminf_{n\in I} h(\cal RF_n)\geq h_0/|\beta|$.
From now on all limits are taken for $n\in I$.

We can conjugate $F_n$ by a real translation and assume that $\Re(Z_n) = 0$.
 
In \Cref{sub:direct} is defined a constant $H_0\geq 0$, the infimum of heights of upper half planes contained in $\Z+L\lambda\dom R$. Is also defined the map $\Lambda = L\lambda -iH_0$.
Let $H_0'=\inf\setof{h>0}{(Z\in\ell\cup U\text{ and }\Im Z>h) \implies Z\in\dom R}$.
We claim that $H_0'\leq y_1$ where
\[y_1 = y_0 + \|J-T_{\beta'_n}\|_\infty + u + \|H-T_{\beta_n}\|_\infty\]
with $u =  (|\beta_n|+\|J-T_{\beta'_n}\|_\infty)\tan \arcsin(\|H-T_{\beta_n}\|_\infty/|\beta_n|)$.
The arguments are similar to \Cref{sub:a0}, when we controlled the domain of $R$ via a constant also called $y_1$: for $\Im Z>y_1$, the point $J(Z)$ is defined and has a forward iterate $Z' = H^{m(Z)}\circ J(Z)$ by $H$ which hits $\ell\cup U$ while staying above $y_0+\|H-T_{\beta_n}\|_\infty$. By definition $R(Z)=Z'$. In particular $H_0'\tend 0$ as $n\to+\infty$.

This implies that $H_0\tend 0$ as $n\to+\infty$: indeed, $\lambda$ tends to the linear map $Z\mapsto Z/\beta$ or $Z\mapsto \ov Z/\beta$ and $L$ tends to the identity on the set of points $Z$ with $\im Z\leq 1$ by \Cref{lem:gluev4}.

The quantity $y_1$ tends to $0$ when $n\to+\infty$. Recall that $\Im Z_n\geq h_0$ which does not depend on $n$. Hence for $n$ big enough we have $\Im Z_n > y_1$. So $Z_n' = R(Z_n)$ is defined.
Note that $Z'_n\notin K(F_n)$ for otherwise $Z_n$ would belong to $K(F_n)$ too.
By \Cref{lem:h2} we have $\Lambda(Z'_n)\notin K(\cal RF_n)$.

Note that $\liminf \Im\lambda(Z_n) \geq h_0/|\beta|$.
We apply \Cref{lem:gluev4} to $L$. Noting $H^\lambda = \lambda\circ H\circ\lambda^{-1}$, we have $|(H^\lambda)'-1|<\delta_n$ on $\H$ and
$H^\lambda\tend T_1$ uniformly on $\H$, so the hypotheses of \Cref{lem:gluev4} are satisfied, with a value of $\delta$ that depends on $\delta_n$ and tends to $0$.
The conclusions of this lemma with $M=1+h_0/|\beta|$ then imply that $\liminf \im L\lambda(Z'_n)\geq h_0/|\beta|$.
(Indeed let $y=1/2+h_0/|\beta|$. For one thing the image of the segment $[iy,H^\lambda(iy)]$ by $L$ followed by the projection $\C\to\C/\Z$ is a closed curve that tends to a horizontal curve as $n\to+\infty$. Points on or above the segment are mapped by $L$ to point whose imaginary part is at least the infimum of $\Im Z$ over this closed curve, and this infimum tends to $y$.
Whereas points $Z$ below the segment satisfy $|\Im L(Z) - Z| < B(M) \delta$ and recall that $\delta \tend 0$ as $n\to+\infty$ and that $\liminf \Im\lambda(Z_n) \geq h_0/|\beta|$.)
Since $L\lambda(Z'_n)-iH_0\notin K(\cal RF_n)$ and $H_0\tend 0$ we get $\liminf h(\cal RF_n)\geq h_0/|\beta|$.
\end{proof}

We complement this lemma with the following one, which requires a Lipschitz-type assumption on $\alpha_n\mapsto F_{\alpha_n}$ and also on $\alpha_n \mapsto F'_{\alpha_n}$.
\begin{lemma}\label{lem:nK}
Assume
\[|F_n(Z)-Z-\alpha_n|\leq K|\alpha_n-\alpha|\]
and \emph{a new assumption:}
\[|F'_n(Z)-1|\leq K|\alpha_n-\alpha|.\]
Then for all $n$ big enough we have
\[\sup_{\Im Z>0}|\cal RF_n(Z)-T_{\alpha'_n}(Z)| \leq D K |\alpha'_n-\alpha'|\]
where $\alpha_n'$
is the rotation number of $\cal RF_n$ and $\alpha'$
is its limit. Here $D>1$ is a universal constant.
\end{lemma}
\begin{proof}
According to \Cref{sub:direct} have $\alpha'_n=\beta'_n/\beta_n$ and $\alpha'=\beta'/\beta$.
An elementary computation yields
%, as $n\to+\infty$:
%\[\alpha'_n-\alpha' \sim \frac{\beta q_{k-1} - \beta' q_k}{\beta^2} (\alpha_n-\alpha)\]
%where one should note that the quantity $\beta q_{k-1} - \beta' q_k$ is the difference of two terms of opposite signs, so cannot vanish.
\[\alpha'_n-\alpha' = \frac{(-1)^k}{\beta\beta_n} (\alpha_n-\alpha).\]

We want to apply \Cref{lem:gluev4} with $M=1$ to get information on
$L$. For this we need to estimate $H^\lambda:=\lambda\circ H\circ \lambda^{-1}$.
Note that $\beta_n$ is the rotation number of $H = T^{-p_k} \circ F_n^{q_k}$.
From the first Lipschitz assumption we get that
\begin{equation}\label{eq:HJlem}
\begin{aligned}
|H(Z)-Z-\beta_n|&\leq K|\beta_n-\beta|,
\\ |J(Z)-Z-\beta'_n|&\leq K|\beta'_n-\beta'|,
\end{aligned}
\end{equation}
see \Cref{sub:iter}.
From the second that
\[|H'(Z)-1|\leq d_n:=(1+K |\alpha_n-\alpha|)^{q_k}-1\sim K|\beta_n-\beta|
\]
when $n\to+\infty$ (there is a similar estimate for $J'$ but we will not use it).
The rotation number of $H^\lambda$ is $1$ and 
\[|H^\lambda(Z)-Z-1| \leq K|\beta_n-\beta| / |\beta_n|
,\]
\[ |(H^\lambda)'-1| \leq d_n\sim K|\beta_n-\beta|
.\]
The bound on the derivative of $H^\lambda$ is better than the bound on $H^\lambda$. However we will apply \Cref{lem:gluev4} which only uses a common bound, i.e.\ here: $K|\beta_n-\beta|/|\beta_n|$, since for $n$ big enough, we have $|\beta_n|<1$.
By this lemma applied to $M=1$ we get
\begin{equation}\label{eq:Lpr}
\Im W\leq 1\ \implies\ |L(W)-W| \leq B(1) K |\beta_n-\beta|/|\beta_n|
\end{equation}
and
\begin{equation}\label{eq:Lpr2}
\Im W\leq 1\ \implies\ |L^{-1}(W)-W| \leq B(1) K |\beta_n-\beta|/|\beta_n|.
\end{equation}

Note that, as $n\to+\infty$:
\[|\beta_n-\beta|/|\beta_n|\sim q_k|\alpha_n-\alpha|/|\beta|\]

Now for $Z\in \dom R$ the return map is $R(Z) = H^{m(Z)}\circ J(Z)$ for some $m(Z)\in\N$.
Let
\[\beta''_n(Z) = \beta'_n+m(Z)\beta_n,\]
\[\beta''(Z)= \beta'+m(Z)\beta.\] 
Then $|R(Z)-(Z+\beta''_n(Z))|\leq K|\beta'_n-\beta'| + K m(Z) |\beta_n-\beta| =^* K|\beta'_n - \beta' + m(Z) (\beta_n-\beta)| = K|\beta''_n(Z)-\beta''(Z)|$, in short
\begin{equation}\label{eq:estR}
|R(Z)-(Z+\beta''_n(Z))| \leq K|\beta''_n(Z)-\beta''(Z)|\
\end{equation}
where equality (*) comes from the fact that $\beta'_n-\beta'$ and $\beta_n-\beta$ have the same sign (and $m(Z)\geq 0$), which may sound surprising since $\beta'_n$ and $\beta_n$ have opposite signs ($\beta'$ and $\beta$ too), so we justify this by the following explicit computation:
$\beta_n'-\beta' = (q_{k-1} \alpha_n-p_{k-1}) - (q_{k-1}\alpha -p_{k-1}) = q_{k-1}(\alpha_n-\alpha)$, and $\beta_n-\beta = (q_k\alpha_n - p_k) - (q_k\alpha -p_k) = q_k(\alpha_n-\alpha)$. Also:
\[\beta''_n(Z)-\beta''(Z) = (q_{k-1}+m(Z)q_k)(\alpha_n-\alpha).\]

We will need a (rough) bound on $m(Z)$: we treat the case $\beta>0$, the other one is symmetric and yields the same bound. Consider all $n$ big enough so that $|\beta_n|-|\beta_n-\beta|<|\beta|/2$ and $|\beta'_n|-|\beta'_n-\beta'|<|\beta'|/2$.
For such an $n$, by \cref{eq:HJlem}, for $Z\in \ell\cup U$ with $Z\in\dom J$, the map $J$ shifts the real part of $Z$ in the negative direction by at most $3\beta'/2$ and the map $H$ of at least $\beta/2$ in the positive direction and at most $3\beta/2$. Since $Z\in\ell\cup U$ we get $\re(Z) \in [0,3\beta/2]$.
Since $m(Z)$ is the first $m\geq 0$ such that $\Re(H^m J(Z))\in\ell\cup U$, it follows that $\Re(H^{m} J(Z)) < 3\beta/2$, and by the above remarks $\Re(H^{m} J(Z)) > m \beta/2+3\beta'/2$ (recall that $\beta'<0$) so
\[m = m(Z) < 3 - 3\beta'/\beta = 3 (1+|\beta'/\beta|).\]
In particular :
\[|\beta''_n(Z)-\beta''(Z)| \tend 0\]
uniformly w.r.t.\ $Z$ as $n\to+\infty$.

Now let $Z^\Lambda \in \H$ and, as in \Cref{sub:direct}, let $k=k(Z^\Lambda)\in\Z$ bet the unique integer such that $Z^\Lambda+k\in \Lambda(\ell\cup U)$ and define $Z = \Lambda^{-1}(Z^\Lambda+k)$. Recall that we defined there $m=m(Z)\in\N$ such that $R(Z) = H^m \circ J(Z)$ and that
\[\cal RF_n(Z^\Lambda) = \Lambda R(Z) -k-m.\]
Recall also that $\Lambda = L\lambda-iH_0$.

We now proceed to the estimate:
\[\begin{aligned}
\cal RF_n(Z^\Lambda) & = L\lambda R(Z) - iH_0 -k-m 
\\ & = \big(L\lambda R(Z) - \lambda R(Z)\big)  + \big(\lambda R(Z) - \lambda(Z+\beta''(Z))\big) 
\\ & + \lambda(Z+\beta''_n(Z))  -iH_0 -k -m.
\end{aligned}\]
And since $\lambda(X+iY) = X/\beta_n +iY/|\beta_n|- iy_0$ we get
\[\begin{aligned} \lambda(Z + \beta''_n(Z)) & = \lambda(Z) + \beta''_n(Z)/\beta_n
\\ & = \lambda(Z) + \beta'_n/\beta_n + m 
\\ & = \lambda(Z) + \alpha'_n + m.
\end{aligned}\]
From this and $Z^\Lambda + k = \Lambda(Z) = L\lambda(Z) - iH_0$
we get
\[\begin{aligned}
\cal RF_n(Z^\Lambda) - (Z^\Lambda +\alpha'_n)& = \big(L\lambda R(Z) - \lambda R(Z)\big)  + \big(\lambda R(Z) - \lambda(Z+\beta''(Z))\big) 
\\ & + \big(\lambda(Z) - L\lambda(Z)\big) .
\end{aligned}\]
Using the estimates above we get:
\[ \text{If }\im \lambda(Z)\leq 1\text{ and }\im \lambda R(Z)\leq 1\text{ then}\]
\[\begin{aligned}|\cal RF_n(Z^\Lambda) - (Z^\Lambda+\beta'_n/\beta_n)|
& \leq B(1)K|\beta_n-\beta|/|\beta_n| + \frac{1}{|\beta_n|} K |\beta''_n(Z)-\beta''(Z)|
\\ & +  B(1)K|\beta_n-\beta|/|\beta_n|
\end{aligned}\]
i.e.\ 
\[|\cal RF_n(Z^\Lambda) - \big(Z^\Lambda+\alpha'_n)| \leq u_n\]
with
\[ u_n:= \frac{K}{|\beta_n|}\big(2B(1)|\beta_n-\beta| + |\beta''_n(Z)-\beta''(Z)|\big).\]
Let us now estimate $u_n$, using the equivalents mentioned in the present proof.
We use $a_n\lesssim b_n$ on non-negative sequences to mean $\exists c_n\geq 0$ such that for $n$ big enough, $a_n\leq c_n$ and $c_n\sim b_n$.
\[ \begin{aligned}
 u_n \sim &\ \frac{K}{|\beta|} (2B(1) q_k + q_{k-1} + m(Z) q_k) |\alpha_n-\alpha| 
\\ \lesssim &\ \frac{K}{|\beta|} (2B(1) q_k + q_{k-1} + 3(1+|\beta'/\beta|) q_k) |\alpha_n-\alpha|.
\end{aligned}\]
And using the comparison between $\alpha'_n-\alpha'$ and $\alpha_n-\alpha$ given at the beginning of the present proof:
\[|\alpha'_n-\alpha'| = \frac{1}{|\beta\beta_n|} |\alpha_n-\alpha|\]
so using $|\beta|\leq 1/q_{k+1}$ and $|\beta'|\leq 1/q_k$: 
\[\begin{aligned}
\frac{u_n}{|\alpha'_n-\alpha'|} \lesssim &\ K|\beta|\left(q_{k-1} + 2B(1)q_k+3(1+|\beta'/\beta|) q_k\right)
%\\ \lesssim &\ K\frac{\left(q_{k-1} + \big(2B(1)+3(1+|\beta'/\beta|)\big) q_k\right)}{q_{k+1}}.
\\ \lesssim &\ K\frac{q_{k-1} + 2B(1)q_k}{q_{k+1}} + 3K(\frac{q_k}{q_{k+1}}+1).
\\ \lesssim &\ (2B(1)+7)K
\end{aligned}\]
Finally: we proved that $\exists c_n$ such that for $n$ big enough, then for all $Z^\Lambda\in \H$ satisfying*
\begin{equation}\label{eq:cond}
\im \lambda(Z)\leq 1\text{ and }\im \lambda R(Z)\leq 1
\end{equation}
we have
\[\frac{|\cal RF_n(Z^\Lambda) - (Z^\Lambda+\alpha'_n)|}{|\alpha'_n-\alpha'|} \leq c_n \sim (2B(1)+7)K.\]
(*): Where $Z$ depends on $Z^\Lambda$ in the way described earlier in the present proof.

We claim that the inequality above extends to all values of $Z^\Lambda\in\H$ by the maximum principle. Indeed the difference $\cal RF_n(Z^\Lambda) - Z^\Lambda-\alpha'_n$ is $\Z$-periodic and is bounded as $\im Z^\Lambda\to +\infty$ because it tends to $0$, so it is enough to prove that $\cal A$ contains the intersection of $\H$ with a neighborhood of $\R$ in $\C$,  where $\cal A$ denotes the set of $Z^\Lambda$ for which \cref{eq:cond} is satisfied.
Recall that $Z=\Lambda^{-1}(Z^\Lambda+k)$, i.e.\ $\lambda(Z)=L^{-1}(Z^\Lambda+k+iH_0)$.
Since $H_0\tend 0$ as $n\to+\infty$ we can assume that $H_0<1/4$.
By \cref{eq:Lpr2}, for $n$ big enough we have $|L^{-1}(W)-W| <1/4$ for all $W\in\dom L$ such that $\Im(W)<1$.
So
\[\text{for }\im Z^\Lambda< 1/4\]
we get $\im (Z^\Lambda + k +iH_0) \leq 2/4$ thus we can apply the estimate on $L$, so $\im L^{-1}(Z^\Lambda+k+iH_0)\leq 3/4$, i.e. $\im \lambda(Z)\leq 3/4$.
Now from \cref{eq:estR} we get $\im R(Z)\leq \im Z + K|\beta''_n(Z)-\beta''(Z)|$ whence
$\im \lambda R(Z)\leq  \im \lambda Z + K|\beta''_n(Z)-\beta''(Z)|/|\beta_n|$,
so
\[\im \lambda R(Z)\leq 3/4 + |\beta''_n(Z)-\beta''(Z)|/|\beta_n|\]
and we have already seen that the right hand side of the sum tends to $0$ uniformly w.r.t.\ $Z$ as $n\to+\infty$. Thus for $n$ big enough we have
$\im \lambda R(Z)\leq 1$.
\end{proof}

The first lemma above implies that we gain a factor $|\beta|$ in estimates of the size of the linearization domain, but by the second lemma we loose a universal factor $D$ in the Lipschitz constant for $\alpha\mapsto F_\alpha$. Moreover this second lemma requires an assumption on the Lipschitz constant for $\alpha\mapsto F'_\alpha$. See later for how we deal with this.

\subsection{Proof of \Cref{lem:main2} for $\alpha=p/q$}\label{sub:apq}

Here we use the results of \Cref{sub:improve} to transfer the case $\alpha=0$ covered in \Cref{sub:a0} to the case $\alpha=p/q$. In the process, the estimate will improve for big values of $q$.

Let us consider one of the two continued fraction expansions of $p/q$ and write it as follows:
\[p/q=[a_0;a_1,\ldots,a_{k+1}].\]
We have $k\geq 0$ since $p/q\notin \Z$. Let $p'/q' = [a_0;a_1,\ldots,a_{k}]$ be its last convergent before $p/q$ itself. Then $p'q-pq' = (-1)^{k+1}$.
In \Cref{lem:main2}, which we are proving, is defined the sequence $\alpha_n= [a_0;a_1,\ldots,a_{k+1},n+1+\sqrt{2}]$ (note that we shifted the index $k$ by one, to match with the notation of \Cref{sub:direct}).
We have $\ds\alpha_n=\frac{p+p'x_n}{q+q'x_n}$ with $x_n=1/(n+1+\sqrt{2})$.
It is important to note that, though $n\to+\infty$, the numbers $q$ and $q'$ remain fixed here.

We now proceed to the order $k+1$ direct renormalization as described in \Cref{sub:direct}.
This yields maps $\cal RF_n$.
Let $\beta = q_k\alpha-p_k$ be the quantity introduced in \Cref{sub:improve}. Here
$\beta = q'\alpha-p' = \frac{q'p-p'q}{q}$ hence
\[\beta = (-1)^k/q.\]
To apply \Cref{lem:nK} we need to control not only the distance from $F_n$ to the rotation but also the distance from $F_n'$ to the constant function $1$.
For this we just apply the Schwarz-Pick inequality:
\[|F_n'(Z)-1| \leq \frac{\sup |F_n-T_{\alpha_n}|}{2\Im Z}\leq \frac{K}{2\im Z} |\alpha_n-\alpha|\]
We restrict $F_n$ to $\Im z>\eps$ for some $\eps\in(0,1/2)$ and then conjugate by the translation by $-i\eps$ to make the domain equal to $\H$. This yields maps $\wt F_n$.
We can apply \Cref{lem:nK} to $\wt F_n$ with the constant $K$ replaced by $K/2\eps$ because the control on $F_n'$ is not as good as the control on $F_n$. The lemma gives that the maps $\cal R\wt F_n$ satisfy the hypotheses of \Cref{lem:main2} with a Lipschitz constant of $DK/2\eps$ where $D>1$ is a universal constant.
Their rotation number $\alpha'_n$ is equal to $\alpha'_n = x_n$, which tends to $\alpha'=0$. So by the case $\alpha=0$ of \Cref{lem:main2} covered in \Cref{sub:a0} we get that
\[\limsup h(\cal R \wt F_n) \leq C(DK/2\eps) = C_0 + \frac{1}{2\pi}\log\frac{DK}{2\eps}\]
for some $C_0>0$.
By \Cref{lem:improve}
\[\limsup h(\wt F_n) \leq \frac{1}{q}\limsup h(\cal R\wt F_n)\]
and since 
\[\limsup h(F_n) \leq \eps + \limsup h(\wt F_n)\]
we get
\[\limsup h(F_n) \leq \eps + \frac{C_0}{q} + \frac{1}{2\pi q} \log \frac{DK}{2\eps}.\]
Optimizing the choice of $\eps\in(0,1/2)$ we get
\[\limsup h(F_n) \leq \frac{C_0}{q} + \frac{1}{2\pi q} \left( 1 + \log (DK\pi q)\right).\]

In this proof we iterated renormalization: we did an order $k+1$ direct renormalization followed by an implicit order $1$ renormalization when using the result of \Cref{sub:a0}.
In fact, order $k+1$ direct renormalization and iterating $k+1$ times order $1$ renormalization are closely related procedures, so morally we could consider we did $k+2$ renormalizations.

\subsection{Proof of  \Cref{lem:main2} for $\alpha\in\R\setminus \Q$}\label{sub:airr}

It will be carried out in two steps.

\smallskip

Recall that, denoting $\alpha = [a_0;a_1,a_2,\ldots]$ we defined
\[\alpha_n = [a_0;a_1,\ldots,a_n,1+a_{n+1},1+\sqrt{2}].\]

\begin{enumerate}
\item
We first prove a weak version of the lemma.
For this we use a first direct renormalization at order $n+1$ for $\alpha_n$, which brings the rotation number $\alpha_n$ of $F_n$ to $\sqrt 2\bmod\Z$ for $\cal RF_n$.
\item
Then, if necessary, we enhance the weak version using a prior renormalization of the type of \Cref{sub:improve}, at some fixed but high order.
\end{enumerate}

So let us apply order $n+1$ direct renormalization to $F_n$ as described in \Cref{sub:direct}.
Be careful with the notations: what is called $\alpha$ in that section is called $\alpha_n$ here, and the integer $k$ in that section is so that $k=n$.
%, in particular we assume that $n\geq 1$ (anyway here we are only interested at what happens when $n\to+\infty$).
A pair of maps is introduced, which we recall:
\begin{eqnarray*}
J &=& T^{-p_{n-1}}\circ F^{q_{n-1}},\\
H &=& T^{-p_{n}}\circ F^{q_{n}} .
\end{eqnarray*}
Their respective rotation numbers are
\begin{eqnarray*}
\beta_n' &=&  q_{n-1}\alpha_n-p_{n-1} ,\\
\beta_n &=& q_{n}\alpha_n-p_{n} .
\end{eqnarray*}
Then a height $y_0$ must be provided satisfying conditions that we recall too:
\begin{itemize}
\item the domain of definition of $H$ contains $\Im Z>y_0$, and hence the domain of $J$ also does;
\item $\forall Z\text{ with }\Im Z>y_0$:
\begin{align*}
\qquad |H(Z)-Z-\beta_n|&\leq |\beta_n|/10\text{,}
& |H'(Z)-1|&\leq 1/10,
\\
\qquad |J(Z)-Z-\beta_n'|&\leq |\beta_n|/10
\text{,}
& |J'(Z)-1|&\leq 1/10,
\end{align*}%
%\item both $\ell\cup U$ and $J(\ell\cup U)$ sit above height $y_0$, where $\ell=i(y_1,+\infty)$ and $U$ is the open strip delimited by $\ell\cup[iy_1,H(i y_1)]\cup H(\ell)$.
\end{itemize}

%Concerning $y_1$, note that the choice
%\[y_1 = y_0 + 2|\beta_n|/10\]
%always works and it will be enough for our aims because $|\beta_n| \tend 0$ as $n\to+\infty$.

Let us proceed to some estimates on rotation numbers.
According to \Cref{sub:remind} and some expression manipulation
\[\beta_n = \frac{(-1)^n}{q_{n+1}+q_{n}\sqrt{2}}.\]
Also,
\[\alpha_n-\alpha = (-1)^{n+1} \frac{\sqrt{2}-x_n}{(q_{n+1}+q_n\sqrt{2})(q_{n+1}+q_n x_n)},\]
where
\[x_n:=[0;a_{n+2},a_{n+3},\ldots]\in(0,1).\]

Let
\[M = q_n K |\alpha_n-\alpha| < K\frac{q_n}{q_{n+1}^2}\sqrt 2\]
and note that $M\tend 0$ as $n\to+\infty$.
The map $H$ is defined at least on $\Im Z> M$ and satisfies there that $H$ differs from $T_{\beta_n}$ by at most $M$.
However the inequality $M\leq |\beta_n|/10$ does not necessarily hold.
By the above
\[ M/|\beta_n| = K \frac{q_n (\sqrt2 -x_n)}{q_{n+1}+q_n x_n},\]
so
\[ \frac{ K q_n}{q_{n+1}} \cdot \frac{\sqrt 2-1}{2} \leq M/|\beta_n| \leq \frac{ K q_n }{q_{n+1}} \sqrt 2.\]
The quotient $q_{n}/q_{n+1} = q_n/(a_{n+1}q_n+q_{n-1})$ is less than one and can be very small if $a_{n+1}$ is big, but it can also be very close to $1$ if $a_{n+1}=1$, depending on the continued fraction expansion of $\alpha_n$. 

Now since $H$ commutes with $T_1$ and $H(Z)-Z$ tends to $\beta_n$ as $\im Z\tend+\infty$, we can improve the estimate on $H$ as follows:
\[\im Z> M \implies |H(Z)-T_{\beta_n}(Z)|\leq e^{-2\pi(\Im Z-M)} M.\]
Hence in all cases the inequality $|H(Z)-T_{\beta_n}(Z)|\leq |\beta_n|/10$ will hold if $\Im Z> y_0$ with
\[y_0 = M +  \frac{1}{2\pi} \log^+ \frac{10M}{|\beta_n|}\]
denoting $\log^+ x = \max(0,\log x)$.
We have
\[y_0\leq M + \frac{1}{2\pi} \log(10K\sqrt2).\]

A similar analysis holds for $J$ with \emph{better} estimates, so we can just take the same constants as above.

The rotation number of $\cal R F_n$ is $-(a_{n+1} + \sqrt 2)\equiv -\sqrt 2 \bmod \Z$. By  \Cref{lem:s2}, $K(\cal R F_n)$ contains $\Im Z > C_{\sqrt 2}$.
We claim that the return map $R$, see \Cref{sub:direct}, is defined on $(\ell\cup U)\cap``\Im Z > y_1"$ with
\[y_1 := y_0 + \frac{3}{10}|\beta_n| + (|\beta_n'| +\frac{1}{10}|\beta_n|)\tan \theta\]
where $\theta =\arcsin (1/10)$.

\begin{proof}
We justify it in the case $\beta_n>0$, the other case being completely similar.
If $\beta_n>0$ then $\beta'_n<0$.
We have $\re J(Z) \geq \beta'_n - \beta_n/10 = -(|\beta'_n|+|\beta_n|/10)$.
By the cone property (see the paragraph between \cref{eq:oneH} and \Cref{lem:fdf2}), the $H$-orbit stays in a cone of apex $J(Z)$ and half opening angle $\theta$ and central axis $J(Z)+\R_+$, as long as it remains in $\Im Z>y_0$.
The condition $\Im Z> y_1$ ensures that $\Im J(z) > y_1-|\beta_n|/10$ and that the orbit will stay above $y_0$ as long as it has not passed the imaginary axis. Before passing it it stays above height $y_1-|\beta_n|/10-(|\beta_n'| +\frac{1}{10}|\beta_n|)\tan \theta = y_0 + \frac{2}{10}|\beta_n|$.
It will pass the imaginary axis (because the real part increases by a definite amount) and when it does, the imaginary part will be at least $y_0 + |\beta_n|/10$, which ensures that it will belong to $\ell\cup U$ (the argument is similar to the paragraph marked (*) on page~\pageref{star1}).
\end{proof}

Let $\lambda$, $L$, $H_0$ and $\Lambda$ be as in \Cref{sub:direct}.
We have $\lambda(X+iY) = X/\beta_n+i(Y-y_0)/|\beta_n|$ and $\Lambda = L\lambda -iH_0$.
We claim that every point $Z$ in $(\ell\cup U) \cap ``\Im Z> y_2"$ with
\[ y_2 := y_1 + (A\max(C_{\sqrt 2},C_1)+C_1+\frac{1}{10})|\beta_n|
\]
is mapped by $\Lambda$ to a point of imaginary part $> C_{\sqrt 2}$,
where $A$ and $C_1$ are the constants in \Cref{lem:gluev3}.

\begin{proof}
Consider the set $\cal A = \Z+L\lambda(\dom R)\subset \H$: it follows from the definition of $H_0$ that $H_0 = \sup_{Z\in\partial \cal A} \im Z$ where the boundary is relative to $\C$.
Let $w'\in \partial \cal A$ with $\im w'=H_0$. By subtracting an integer to $w'$ we can assume that $w'\in L\lambda(\ell\cup U)$.
Assume now that $\im Z>y_2$ defined above.
We will apply \Cref{lem:gluev3} to $W=\lambda(Z)$ and $W'=L^{-1}(w') \in \lambda (\ell\cup U)$.

Let $Z' = \lambda^{-1}W' \in \ell\cup U$.
If $Z'\in U$ then we have $\im \lambda^{-1} W' \leq y_1$ for otherwise $Z'$
would have a neighbourhood $V$ contained in $\dom R$
hence its image by $L\lambda$ would belong to the interior of $\cal A$, contradicting the definition of $w'$.
For similar reasons, if $Z' \in \ell$ then we have $\im Z' \leq y_1 + |\beta_n|/10$: recall that the map $L$ extends to a neighbourhood of $\ell^\lambda = \lambda(\ell)$ and a neighbourhood of $H^\lambda(\ell^\lambda)$ and satisfies $L\circ H^\lambda = L +1$ near $\ell^\lambda$.
The margin $|\beta_n|/10$ is there to ensure that both $V$ and $H(V)$ are above $y_1$ for a small enough neighborhood $V$ of any $Z$ with $\Im Z>y_1+|\beta_n|/10$. Hence $\Im Z'\leq y_1+|\beta_n|/10$ in all cases.

Hence $\im W - \im W' > (y_2-y_1-|\beta_n|/10)/|\beta_n| = A\max(C_{\sqrt 2},C_1)+C_1$.
\Cref{lem:gluev3} gives $|\im L(W)-\im L(W')|\geq (|\Im W -\im W'| -C_1)/A > \max(C_1,C_{\sqrt 2})$.
In particular $|\im L(W)-\im L(W')| > C_1$ so by the second part of \Cref{lem:gluev3}, the quantities $\im L(W)-\im L(W')$ and $\im W-\im W'$ have the same sign.
We thus get that $\im \Lambda(Z) = \im L\lambda(Z) - H_0 = \im L(W) - \im L(W') > C_{\sqrt 2}$.
\end{proof}

Under these conditions on $Z$, it follows that $\Lambda(Z)\in K(\cal RF_n)$. Hence $Z\in K(F_n)$.
The same analysis can be applied to the conjugate of $F_n$ by a horizontal translation $T_x(Z)=Z+x$. 
Write $\ell_x$ and $U_x$ the sets constructed from $T_x^{-1}\circ F_n\circ T_x$ instead of $F_n$. (It turns out that $\ell_x = (iy_0,+i\infty) = \ell$ is independent of $x$.)
For any point $Z\in \H$ with $\im Z>y_2$, there is a translation $T_x$ so that $T_x^{-1} Z\in \ell_x \cup U_x$ (in fact take $x=\re Z$, then $T_x^{-1} Z\in\ell=\ell_x$). Hence $T_x^{-1} Z\in K(T_x^{-1}\circ F_n\circ T_x)$, which is equivalent to the statement $Z\in K(F_n)$. It follows from this analysis that
\[h(F_n)\leq y_2.\]
The quantity $y_2$ depends on $n$ and we have $y_2-y_0\tend 0$ hence  $\limsup y_2 = \limsup y_0 \leq \frac{1}{2\pi}\log(10 K\sqrt 2)$.

As a consequence we have proved the following (weak) asymptotic estimate
\begin{equation}\label{eq:weak}
\limsup_{n\to+\infty} h(F_n)\leq \frac{1}{2\pi}\log(10K\sqrt2).
\end{equation}
The constant $\sqrt 2$ here has nothing to do with our choice of rotation numbers involving $\sqrt 2$.

We now enhance this estimate by a prior renormalization of fixed---yet high---order. 
More precisely we temporarily fix some $k\geq 0$ and $\eps>0$.
Let $\wt F_\eps$ be the map obtained by conjugating $F_\eps$ by the translation by $-i\eps$ and then restricting to $\H$.
Then $h(F_n)\leq \eps + h(\wt F_n)$.
By the Schwarz-Pick inequality we have
\[|F_n'(Z)-1|\leq \frac{K|\alpha_n-\alpha|}{2\im Z},\]
and this implies
\[\sup_{\H}|\wt F_n'-1|\leq \frac{K}{2\eps} \big|\alpha_n-\alpha\big|.\]
Now for $n>k$ let $\cal RF_n$ be the order $k+1$ direct renormalization of $\wt F_n$ provided by \Cref{lem:improve}.
According to this lemma,
\[\limsup h(\wt F_n)\leq |\beta|\limsup h(\cal R\wt F_n)\]
where $\beta = \beta(k) = q_k\alpha-p_k$ (recall that for fixed $k$, the first $k$ convergents of $\alpha$ and $\alpha'$ coincide for large enough $n$).
Now by \Cref{lem:nK} with $K$ replaced by $K/2\eps$ (we assume $\eps<1/2$), the sequence $\cal R\wt F_n$ satisfies
\[\sup_{\Im Z>0}|\cal R\tilde F_n(Z)-T_{\alpha'_n}(Z)| \leq \frac{D K}{2\eps} |\alpha'_n-\alpha'|\]
where $\alpha'_n$ and $\alpha'$ are the respective rotation numbers of $\cal R\wt F_n$ and of its limit.
We have $\alpha' = -[a_{k+1},a_{k+2},\ldots]$
and $\alpha'_n = -[a_{k+1},\ldots,a_n,1+a_{n+1},1+\sqrt 2]$. We are thus in the situation of \Cref{lem:main2} for $-\alpha'$ in place of $\alpha$ and $s\circ(\cal R\tilde F_n)\circ s$ in place of $F_n$, where $s(X+iY)=-X+iY$.
By the weak estimate above \Cref{eq:weak} we thus have
\[
\limsup_{n\to+\infty} h(\cal R\wt F_n)\leq \frac{1}{2\pi}\log\frac{5DK\sqrt2}{\eps}
\]
and thus
\[
\limsup_{n\to+\infty} h(F_n)\leq \eps + |\beta(k)| \frac{1}{2\pi}\log\frac{5DK\sqrt2}{\eps}.
\]
Now this is valid for all $k>0$ and since $\beta(k)\tend 0$ as $k\to+\infty$ and neither $D$, $K$, nor $\eps$ depend on $k$, we get
\[
\limsup_{n\to+\infty} h(F_n)\leq \eps.
\]
Since this is valid for all $\eps \in (0,1/2)$ we conclude:
\[\limsup_{n\to+\infty} h(F_n)\leq 0.\]
This ends the proof of \Cref{lem:main2}.

For this case, as in the case $\alpha=p/q$, we used a direct renormalization of a direct renormalization, though in a more subtle way.

\appendix

\section{Analytic degenerate families}\label{app}

An obvious way of obtaining degenerate families is to conjugate the family of rigid rotations $R_\alpha(z)=e^{2\pi i\alpha}$ by a family of varying analytic diffeomorphisms. The next lemma shows that in the case of families depending analytically on $\alpha$, this is the only way.

\begin{proposition}
Let $I$ be an open subset of $\R$. Assume $\{f_\a:\D\to
\C\}_{\a\in I}$ is an $\R$-analytic family of maps which fix $0$
with multiplier $e^{2i\pi\a}$. The following are equivalent:
\begin{enumerate}
\item\label{item:p1} the family $\{f_\a\}_{\a\in I}$ is degenerate;
\item\label{item:p2} there exist an open interval $J\subset I$, a real
$\delta>0$ and an analytic map $\phi:J\times B(0,\delta)\to \C$
such that for all $\a\in J$, $\phi_\a(z)=z+{\cal O}(z^2)$ and for
all $z\in B(0,\delta)$, $f_\a = \phi_\a\circ R_\a \circ \phi_\a^{-1}$
(with $\phi_\a= \phi(\a,\cdot)$).
\end{enumerate}
\end{proposition}

\begin{proof} (\ref{item:p2}) $\implies$ (\ref{item:p1}).  Obvious.

(\ref{item:p1}) $\implies$ (\ref{item:p2}).  Let $U$ be a domain intersecting $\R$ in an
interval $J$ contained in $I$ such that $f_\a$ is defined for all
$\a \in U$ and is linearizable for every $\a \in J \cap \Q$. Let
$\phi_\a$, $\a \in U \setminus \Q$ be the (uniquely defined)
formal linearization of $f_\a$, so $\phi_\a$ is a formal power
series
$$\phi_\a(z)=z+\sum_{n=2}^\infty a_n(\a) z^n$$
satisfying
\begin{equation}\label{eq:lin}
\phi_a \circ R_\a = f_\a \circ \phi_a
\end{equation}
 formally.  The $a_n$
can be found recursively from the power series expansion of
$f_\a$ (see \cite{Pf} or below), and from the formula one obtains it follows that they are, in general, meromorphic functions of $\a \in U$, with possible poles when the multiplier is a root of unity of order $\leq n$, i.e.\ when $\alpha =p/q$ with $1\leq q\leq n$.

Recall that we assumed that $f_\a$ is linearizable for all $\a\in J\cap \Q$. Let us prove that this implies that the functions $a_n(\a)$ have no poles.
\begin{lemma}\label{lem:b1}
The function $\alpha\mapsto a_n(\alpha)$ has a holomorphic extension to $U$.
\end{lemma}
\begin{proof}
We will proceed by induction on $n$.
We have $a_1(\alpha)=1$, which initializes the recurrence.
Let $n>1$ and assume that for all $k<n$, the map $\alpha\mapsto a_k(\alpha)$ has a holomorphic extension to $U$. Let $\rho(\a)= e^{i 2\pi \alpha}$ be the multiplier and let $b_n(\a)$ such that $f_\a = \rho(\a)+\sum_{m\geq 2} b_m(\a) z^m$. 
\Cref{eq:lin} then reads
\[
\sum_{i\geq 1} a_i(\a) \rho(\a)^i z^i = \rho(\a)\sum_{i\geq 1} a_i(\a) z^i + \sum_{m\geq 2} b_m (\sum_{i\geq 1} a_i(\a) z^i)^m
\]
In particular for the coefficient in $z^n$:
\[
\rho(\a)^n  a_n(\a)= \rho(\a) a_n(\a) + \sum_{m=2}^{n} b_m(\a) \sum_{\ldots} a_{i_1}(\a)\cdots a_{i_m}(\a),
\]
where the $\sum_{\ldots}$ is over all the $m$-uplets of positive integers whose sum are equal to $n$.
We rewrite the last line as follows:
\begin{equation}
\label{eq:lin2} (\rho(\a)^n-\rho(\a))  a_n(\a) = P_{b(\a),n}(a_2(\a),\ldots,a_{n-1}(\a))
\end{equation}
where $ P_{b(\a),n}(x_2,\ldots,x_{n-1}) = \sum_{m=2}^{n} b_m(\a) \sum_{\ldots} x_{i_1}\cdots x_{i_m}$.
%\[ a_n(\a) = \frac{\sum_{m=2}^{n} b_m(\a) \sum_{\ldots} a_{i_1}(\a)\cdots a_{i_m}(\a)}{\rho(\a)^n-\rho(\a)}.\]
By the induction hypothesis, the right hand side of \cref{eq:lin2} is holomorphic, hence the function $a_n(\a)$ has \emph{at most simple poles}, situated at $\alpha =p/(n-1)$, $p\in\Z$. If we prove that the right hand side of \cref{eq:lin2} vanishes for $\a=p/(n-1)$, then we will have completed the induction.
Now we must be careful: by assumption for every $p/q$ there is a solution $\phi$ to $\phi \circ R_{p/q} = f_{p/q}\circ \phi$; however this solution is not unique.
\begin{sublemma}\label{sublem:one}
If $\zeta$ is a formal power series such that $\zeta \circ R_{p/q} - R_{p/q} \circ \zeta = \cal O(z^m)$ and if $q|(m-1)$ then $\zeta \circ R_{p/q} - R_{p/q} \circ \zeta = \cal O(z^{m+1})$.
\end{sublemma}
\proof By a straightforward computation, for \emph{any} formal power series,
all the coefficients of $\zeta \circ R_{p/q} - R_{p/q} \circ \zeta$ with order in $1+q\Z$ vanish.
\ \hfill$\blacksquare$\par\medskip
\begin{sublemma}\label{sublem:two}
Let $f(z) = e^{2\pi ip/q} z + \sum_{n\geq 2} b_n z^n$ be a holomorphic or formal power series, and assume that there is a formal power series $\tilde \phi = z +\sum_{n\geq 2} \tilde a_n z^n$ solution of $\tilde\phi \circ R_{p/q} = f \circ \tilde\phi$ and another formal power series $\phi = z +\sum_{n\geq 2} a_n z^n$ such that $\phi \circ R_{p/q} - f \circ \phi = \cal O(z^m)$ for some $m\geq 2$. This depends only on $a_2$, \ldots, $a_{m-1}$; fix these values and consider the equation $\phi \circ R_{p/q} - f \circ \phi = \cal O(z^{m+1})$ with unknown $a_{m}$. Assume that $p/q$ is in its lowest terms.
Then
\begin{enumerate}
\item\label{item:sl1} if $m-1$ is not a multiple of $q$, there is a unique solution $a_m$;
\item\label{item:sl2} if $m-1$ is a multiple of $q$, all $a_m\in \C$ are solutions: in other words $P_{b,m}(a_2,\ldots,a_{m-1}) = 0$.
\end{enumerate}
\end{sublemma}
\proof
Case (\ref{item:sl1}) is immediate in view of \cref{eq:lin2}. Assume we are in case (\ref{item:sl2}).
The formal power series $\zeta = \phi^{-1} \circ \tilde \phi$ commutes with $R_{p/q}$ up to order $m-1$ included, and thus by \Cref{sublem:one} up to order $m$ included. It follows that $\phi \circ R_{p/q} - f \circ \phi = \cal O(z^{m+1})$.
\ \hfill$\blacksquare$\par\medskip
It follows from Case (\ref{item:sl2}) of the above lemma applied to the reduced form of $p/(n-1)$ that $P_{b(\a),n}(a_2(\a),\ldots,a_{n-1}(\a))=0$.
This cancels the possible simple pole to $a_n(\a)$ at $\a=p/(n-1)$ and proves heredity of the induction hypothesis. \Cref{lem:b1} follows. 
\end{proof}

If $\a=p/q\in\Q$, we let $\phi_a = z+ \sum_{n\geq 2} a_n(\a) z^n$ for the holomorphic extension of the functions $a_n(\a)$ at $a=p/q$.

For $\a\in U$ not necessarily real, let $R(\a)\in [0,+\infty]$ be the radius of convergence of $\phi_a$, $s(\a)\leq R(a)$ be the radius of the maximal disk centered at $0$
around $0$ where $\phi_\a$ takes values in $\D$,
and $r(\a)\leq s(\a)$ the maximal radius disk of such a disk for which moreover $\phi_a$ is injective. It is easy to see that $s(\a)$ is locally bounded away
from zero in $U \setminus J$.\footnote{This also holds for $r(\a)$.}

\begin{lemma}\label{lem:b3}
If for $\a_0 \in J$ we
have $r(\a_0)>0$, then $\ds \liminf_{\eps \to 0} s(\a_0+i \eps)>0$.
\end{lemma}
\begin{proof}  Set
$$g_\eps=\phi^{-1}\circ f_{\a_0+i\eps}\circ
\phi$$
where $\phi$ denotes the restriction of $\phi_{\a_0}$ to $r(\a_0)\D$.
The domain of $g_\eps$ tends to $r(\a_0)\D$ as $\eps\to 0$ and $g_\eps\tend R_{\a_0}$ uniformly locally.
Also,
$$g_\eps(z)=e^{-2\pi \eps+2i\pi \a_0}z+{\cal O}(\eps
z^2).$$
In particular, $|g_\eps(z)| = (1-2\pi\eps) |z| + \cal O(\eps z^2)$.
So, there exists $0<r_0<r(\a_0)$ so that
\begin{itemize}
\item when $|z|<r_0$ and $\eps>0$, $|g_\eps(z)|<|z|$ and \item when
$|z|<r_0$ and $\eps<0$, $|g_\eps(z)|>|z|$. \end{itemize} For $\eps$
sufficiently close to $0$, $g_\eps$ is univalent on $B(0,r_0)$. So,
there is a univalent map $\psi_\alpha:\ B(0,r_0)\to \C$ which conjugates $g_\eps$
to $R_{\a_0+i\eps}$. The map $\tilde\phi := \phi_{\alpha_0} \circ \psi_{\alpha}^{-1}$ satisfies the equation $\tilde\phi\circ R_\alpha = f\circ \tilde\phi$ near $0$ so by uniqueness has the same expansion as $\phi_\alpha$.
It follows from the Koebe One Quarter Theorem applied to $\psi_{\alpha}$ that
as $\eps\to 0$,
$$\liminf r(\a_0+i\eps)\geq r_0/4.$$
Since $s\geq r$, the lemma follows.
\end{proof}

Consider two Brjuno numbers $\a_0<\a_1$ in $J$ (so that $r(\a_0)>0$ and $r(\a_1)>0$) and $y>0$ small enough so that the box $U'$ of equation $\a_0<\Re z<\a_1$ and $|\im z|<y$ is compactly contained in $U$.
By \Cref{lem:b3} we have $s|\partial U' \geq
\delta>0$. By Cauchy's formula applied to $\phi_a$ in the disk $s(\a)\D$, we see that $a_n(\a) \leq \delta^{-n}$ holds for all $s\in\partial U'$ and thus by the maximum principle holds for all $s\in U'$. It follows that $(\a,z) \to \phi_\a(z)$ is defined and
holomorphic $U'\times B(0,\delta) \to \D$. It
satisfies $\phi_\a(z)=z+{\cal O}(z^2)$ and $\phi_\a \circ R_\a=f_\a \circ \phi_\a$ by analytic continuation.
This proves claim (\ref{item:p2}) for $J=(\a_0,\a_1)$.
\end{proof}

\section{General statement}\label{app:grl}

We recall here the main statement in \cite{BC}, adapted it to our situation.

\begin{notation}
Let $X$ and $Y$ be topological spaces and $X\subset Y$. We
write $X\subset_{\on{0}} Y$ if the canonical injection
$X\hookrightarrow Y$ is continuous. If moreover $X$ is a normed vector space and $Y$ a Fréchet space,\footnote{We do not assume that the norm on $X$ and the distance on $Y$ are related.} we write $X\subset_{\on{c}} Y$ if every bounded set in $X$ has compact closure in $Y$.
\end{notation}

In the theorem below we assume, as in most of the present article, that $I\subset \R$ is an open interval and that $f_\alpha:\D\to \C$ is a family of analytic maps that depends continously on $\alpha \in I$, with
\[f_\alpha(z) = e^{2\pi i\alpha}z+\cal O(z^2).\]
Below we use $r(\alpha)$ from \Cref{def:phi} in the present article, and $\phi_\alpha$ from \Cref{def:confrad}. The notation $C^0$ refers to the set of holomorphic maps on $\D$ that have a continuous extension to $\ov \D$, endowed with the sup-norm. The notation $C^\omega$ refers to the set of holomorphic maps on $\D$ that have a holomorphic extension to a neighborhood of $\ov \D$ in $\C$.

\begin{theorem}\label{theo_general}
Let $F$ be any Fréchet space such that $C^\omega \subset F
\subset_{\on{0}} C^0,$ and let \[B \subset_{\on{c}} F\] be a Banach
space.
If the family $(f_\alpha)$ is \emph{non-degenerate} (see \Cref{def:dege}) and the dependence on $\alpha$ is \emph{Lipschitz} then there exists a Brjuno number $\alpha$ such that
\begin{itemize}
\item $\partial\Delta_\alpha$ is compactly contained in $\D$,
\item the map $z \mapsto \phi_{\alpha}(r(\alpha) z)$ belongs to
$F$ but not to $B$.
\end{itemize}
Equivalently, one can replace the Banach space $B\subset_{\on{c}}
F$ by a countable union of Banach spaces $B_n\subset_{\on{c}} F$
or by a countable union of compact sets $K_n\subset F$.
\end{theorem}

See section~1 of \cite{BC} to see how one deduces \Cref{thm:main} from \Cref{theo_general}.

\bibliographystyle{alpha}
\bibliography{bib}

\end{document}